\documentclass[11pt]{article}

\usepackage{fourier} 
\usepackage{soul} 
\usepackage{enumerate} 

\usepackage{amsfonts,amssymb,amsmath,amsthm}
\usepackage{bbm}

\numberwithin{equation}{section}

\def\QQ{\mathbb{Q}}
\def\RR{\mathbb{R}}

\def\EE{\mathbb{E}}
\def\11{\mathbbm{1}}

\def\E{\mathbb{E}}
\def\P{\mathbb{P}}
\def\R{\mathbb{R}}
\def\Q{\mathbb{Q}}
\def\N{\mathbb{N}}

\def\d{\partial}
\def\Z{\mathbb{Z}}

\def\cE{{\cal E}}
\def\cF{{\cal F}}

\newtheorem{thm}{Theorem}[section]
\newtheorem{lem}[thm]{Lemma}
\newtheorem{cor}[thm]{Corollary}

\newtheorem{prop}[thm]{Proposition}

\theoremstyle{remark}
\newtheorem{rem}{Remark}

    \def\restriction#1#2{\mathchoice
                  {\setbox1\hbox{${\displaystyle #1}_{\scriptstyle #2}$}
                  \restrictionaux{#1}{#2}}
                  {\setbox1\hbox{${\textstyle #1}_{\scriptstyle #2}$}
                  \restrictionaux{#1}{#2}}
                  {\setbox1\hbox{${\scriptstyle #1}_{\scriptscriptstyle #2}$}
                  \restrictionaux{#1}{#2}}
                  {\setbox1\hbox{${\scriptscriptstyle #1}_{\scriptscriptstyle #2}$}
                  \restrictionaux{#1}{#2}}}
    \def\restrictionaux#1#2{{#1\,\smash{\vrule height .8\ht1 depth .85\dp1}}_{\,#2}}

\begin{document}

\title{Uniform convergence of penalized time-inhomogeneous Markov processes}

\author{Nicolas Champagnat$^{1,2,3}$, Denis Villemonais$^{1,2,3}$}

\footnotetext[1]{IECL, Universit\'e de Lorraine, Site de Nancy, B.P. 70239, F-54506 Vandœuvre-lès-Nancy Cedex, France}
\footnotetext[2]{CNRS, IECL, UMR 7502, Vand{\oe}uvre-l\`es-Nancy, F-54506, France}  
\footnotetext[3]{Inria, TOSCA team, Villers-l\`es-Nancy, F-54600, France.\\
  E-mail: Nicolas.Champagnat@inria.fr, Denis.Villemonais@univ-lorraine.fr}

\maketitle

\begin{abstract}
We provide an original and general sufficient criterion ensuring the exponential contraction of Feynman-Kac semi-groups of penalized processes. This criterion is applied to time-inhomogeneous one-dimensional diffusion processes conditioned not to hit $0$ and to penalized birth and death processes evolving in a quenched random environment.
\end{abstract}



\section{Introduction}
\label{sec:vrai-intro}

In~\cite{champagnat-villemonais-15}, we developed a probabilistic framework to study Markov processes with absorption conditionned on
non-absorption. The main result is a necessary and sufficient condition for the exponential convergence of conditional distributions
to a unique quasi-stationary distribution. Our approach is based on coupling estimates (Doeblin condition and Dobrushin coefficient)
which allow to use probabilistic methods to check the criteria in various classes of models, such as one-dimensional
diffusions~\cite{champagnat-villemonais-15b,champagnat-villemonais-15c}, multi-dimensional
diffusions~\cite{champagnat-coulibaly-villemonais-16} or multi-dimensional birth and death
processes~\cite{champagnat-villemonais-15d}.

Because our method is general and only makes use of semi-group properties and coupling criteria, its extension to the
time-inhomogeneous setting is natural. Actually, it appears that our method naturally extends to the even more general setting of the
contraction of Feynman-Kac semi-groups of penalized processes developed by Del Moral and Miclo~\cite{delmoral-miclo-02} and Del Moral
and Guionnet~\cite{DelMoral-Guionnet2000}. The literature on the topic is vast and closely related to the study of genealogical and
interacting particle systems. For more details, we refer the reader to the two textbooks~\cite{delmoral-04-book,delmoral-13-book} and
the numerous references therein.

The present paper can be seen as a complement to the results on the contraction of Feynman-Kac semi-groups gathered
in~\cite[Chap.\,12]{delmoral-13-book}. Our results apply both to the discrete-time and continuous-time cases. To show the novelty of
our criteria and how to apply the methods developed
in~\cite{champagnat-villemonais-15,champagnat-villemonais-15b,champagnat-villemonais-15c}, we provide a detailed study of two natural
classes of models that cannot be directly treated using previously known criteria: time-inhomogeneous diffusion processes with hard
obstacles in dimension 1 and time-inhomogeneous penalized one-dimensional birth and death processes. We also consider the case of
birth and death processes evolving in a quenched random environment, alternating phases of growth and decay, under very general
assumptions on the environment.

In Section~\ref{sec:intro}, we present the general class of models we consider and state our main result on the contraction of
Feynman-Kac semi-groups in the general framework of penalized time-inhomogeneous processes (Theorem~\ref{thm:main-result}). We then
obtain in Section~\ref{sec:other-results} new results on the limiting behavior of the expectation of the penalization
(Proposition~\ref{prop:eta}) with consequences on uniqueness on time-inhomogeneous stationary evolution problems with growth
conditions at infinity, and on the existence and asymptotic mixing of the Markov process penalized up to infinite time
(Theorem~\ref{thm:Q-proc}). We apply these results to time-inhomogeneous diffusions on $[0,+\infty)$ absorbed at $0$ and conditioned
on non-absorption (that is, with infinite penalization at 0) in Section~\ref{sec:exa_diff}. Section~\ref{sec:exa_PNM} is devoted to
the study of penalized continuous time inhomogeneous birth and death processes in $\N$: we first give a general criterion in
Subection~\ref{sec:PNM-general} and then study the case of birth and death processes in quenched environment alternating phases of
growth and decay (close to infinity) in Subection~\ref{sec:PNM-ex}. The proof of Theorem~\ref{thm:main-result} is given in
Section~\ref{sec:pf-main-thm}. Proposition~\ref{prop:eta} and Theorem~\ref{thm:Q-proc} are proved respectively in
Sections~\ref{sec:pf-prop} and~\ref{sec:pf-Q-proc}.

\section{Main result}
\label{sec:intro}

Let $\left(\Omega,(\mathcal{F}_{s,t})_{0\leq s\leq t\in I},\P,(X_t)_{t\in I}\right)$ be a Markov process evolving in a measurable
space $(E,\cE)$, where the time space is $I=[0,+\infty)$ or $I=\N$ and $X$ can be time-inhomogeneous, such that $X_t$ is
$\mathcal{F}_{s,r}$-measurable for all $s\leq t\leq r$. Let $Z=\{Z_{s,t};0\leq s\leq t,\ s,t\in I\}$ be a collection of
multiplicative nonnegative random variables such that, for any $s\leq t$, $Z_{s,t}$ is a $\mathcal{F}_{s,t}$-measurable random
variable and
\begin{align}
  \label{eq:hyp-first}
  \EE_{s,x}(Z_{s,t})>0\quad\text{and}\quad\sup_{y\in E}\,\EE_{s,y}(Z_{s,t})<\infty\quad\forall s\leq t\in I\ \forall x\in E.
\end{align}
By multiplicative, we mean that, for all $s\leq r\leq t\in I$,
\begin{align*}
  Z_{s,r}Z_{r,t}=Z_{s,t}.
\end{align*}

We define the non linear semi-group $\Phi=\{\Phi_{s,t};\ 0\leq s\leq t\}$ on the set $M_1(E)$ of all probability measures on $E$ by
setting, for any distribution $\mu\in M_1(E)$, $\Phi_{s,t}(\mu)$ as the probability measure on $E$ such that, for any bounded and $\cE$-measurable function $f:E\rightarrow\R$,
 \begin{align}
   \label{eq:def-phi}
 \Phi_{s,t}(\mu)(f):=\frac{\E_{s,\mu}(f(X_t)Z_{s,t})}{\E_{s,\mu}(Z_{s,t})},
 \end{align}
 where $((X_t)_{t\geq s},\P_{s,\mu})$ denotes the Markov process $X$ on $[s,+\infty)$ starting with initial distribution $\mu$ at time $s$.

Typical examples of penalizations are given by
\begin{align}
\label{eq:ex-penaliz}
Z_{s,t}=\11_{X_t\not\in D}\quad\text{or}\quad Z_{s,t}=e^{\int_s^t \kappa(u,X_u) du},
\end{align}
where $D\subset E$ is some absorbing set for the process $X$ or  $\kappa$ is a measurable function from $\RR_+\times E$ to $\RR$. In
the first case, $\Phi_{s,t}(\mu)$ is simply the conditional distribution of $X_t$ with distribution $\mu$ at time $s$, given it is
not absorbed in $D$ at time $t$. In the second case, if $\kappa(t,x)\leq 0$ for all $t\geq 0$ and $x\in E$, then $-\kappa(t,x)$ can
be interpreted as a killing rate at time $t$ in position $x$ and $\Phi_{s,t}(\mu)$ is the conditional distribution of $X_t$ with
distribution $\mu$ at time $s$, given it is not killed before time $t$. Note that if $\kappa$ is bounded from above by a finite
constant $\bar\kappa$, then we can replace $\kappa$ by $\kappa-\bar\kappa$ without modifying $\Phi_{s,t}(\mu)$ and hence recover the
previous interpretation of $\bar\kappa-\kappa$ as a killing rate.

For all $s\geq 1$ and all $x_1,x_2\in E$, we define the non-negative measure on $E$
\begin{align*}
\nu_{s,x_1,x_2}=\min_{i=1,2} \Phi_{s-1,s}(\delta_{x_i}),
\end{align*}
where the minimum between two measures is understood as usual as the largest measure smaller than both measures, and the real
constant
\begin{align*}
d_s=\inf_{t\geq 0,x_1,x_2\in E} \frac{\E_{s,\nu_{s,x_1,x_2}}(Z_{s,s+t})}{\sup_{x\in E} \E_{s,x}(Z_{s,s+t})}.
\end{align*}
Similarly, we define
\begin{align}
\label{eq:intro-nu-s}
\nu_{s}=\min_{x\in E} \Phi_{s-1,s}(\delta_{x})
\end{align}
and the real constant
\begin{align}
\label{eq:intro-d'-s}
 d'_s=\inf_{t\geq 0} \frac{\E_{s,\nu_{s}}(Z_{s,s+t})}{\sup_{x\in E} \E_{s,x}(Z_{s,s+t})}.
\end{align}
Note that $\nu_s\leq\nu_{s,x_1,x_2}$ and $d'_s\leq d_s$.

Let us define, for all $0 \leq s \leq
 t\leq T$ the linear operator $K_{s,t}^T$ on the set of bounded measurable function on $E$ by
\begin{align}
   \label{eq:def-K}
  K_{s,t}^T f(x)&= \frac{\E_{s,x}(f(X_t)Z_{s,T})}{\E_{s,x}(Z_{s,T})}.
\end{align}
We extend as usual this definition to any initial distribution $\mu$ on $E$ as
\begin{align*}
  \mu K_{s,t}^T f&= \int_E K_{s,t}^T f(x)\,\mu(dx).
\end{align*}
Note that $K^t_{s,t}f(x)=\Phi_{s,t}(\delta_x)(f)$ but $\mu K^t_{s,t}f\neq\Phi_{s,t}(\mu)(f)$ in general.

\begin{thm}
\label{thm:main-result}
For all probability measures $\mu_1,\mu_2$ on $E$ and for all $0\leq s\leq s+1 \leq t\leq T\in I$, we have
\begin{align}
\label{eq:main-1}
\left\|\mu_1 K_{s,t}^T-\mu_2 K_{s,t}^T\right\|_{TV}\leq  \prod_{k=0}^{\lfloor t-s \rfloor-1}\left(1- d_{t-k}\right) \|\mu_1-\mu_2\|_{TV}
\end{align}
and
\begin{align}
\label{eq:main-2}
\left\|\Phi_{s,t}(\mu_1)-\Phi_{s,t}(\mu_2)\right\|_{TV}\leq 2 \prod_{k=0}^{\lfloor t-s \rfloor-1}\left(1-d_{t-k}\right),
\end{align}
where $\|\cdot\|_{TV}$ denotes the usual total variation distance: for all signed finite measure $\mu$ on $E$,
\begin{align*}
  \|\mu\|_{TV}=\sup_{A\in\mathcal{E}}\mu(A)-\inf_{A\in\mathcal{E}}\mu(A).
\end{align*}
\end{thm}

In particular, if $\limsup_{t\rightarrow\infty} d_t>0$, there is convergence in~\eqref{eq:main-1} and~\eqref{eq:main-2} when
$t\rightarrow+\infty$, and if $\inf_{s\in I}d_s>0$, or more generaly if $\limsup_{t\rightarrow\infty}\frac{1}{t}\sum_{s\leq
  t}\log(1-d_s)<0$, we have geometric convergence in~\eqref{eq:main-1} and~\eqref{eq:main-2}. There is also convergence for example
if $d_t\geq ct^{-1}$ for $t$ large enough for some $c>0$.

\begin{rem}
  \label{rem:extension}
  Note that, in the definition of $\nu_{s,x_1,x_2}$ and $d_s$, the time increments of $+1$ are not restrictive, since we could change
  the time-scale in the definition of the time-inhomogeneous Markov process $X$ and the penalization $Z$ using any deterministic
  increasing function. In particular, given $s=s_0<t_0\leq s_1<t_1\leq\ldots\leq s_n<t_n\leq t$ in $I$, we may define for all
  $i=0,\ldots,n$ and all $x_1,x_2\in E$,
  \begin{align*}
    \nu_{s_i,t_i,x_1,x_2}=\min_{j=1,2} \Phi_{s_i,t_i}(\delta_{x_j}),
  \end{align*}
  and the real constant
  \begin{align*}
    d_{s_i,t_i}=\inf_{t\geq 0,x_1,x_2\in E} \frac{\E_{t_i,\nu_{s_i,t_i,x_1,x_2}}(Z_{t_i,t_i+t})}{\sup_{x\in E} \E_{t_i,x}(Z_{t_i,t_i+t})}.
  \end{align*}
  Then it is straightforward to extend the proof of Theorem~\ref{thm:main-result} (this can be obtained using an appropriate time
  change to recover $\nu_s$ and $d_s$) to prove that, for all probability measures $\mu_1,\mu_2$ on $E$ and all $T\geq t$, we have
  \begin{align*}
    \left\|\mu_1 K_{s,t}^T-\mu_2 K_{s,t}^T\right\|_{TV}\leq  \prod_{k=0}^{n}\left(1- d_{s_k,t_k}\right) \|\mu_1-\mu_2\|_{TV}
  \end{align*}
  and
  \begin{align*}
    \left\|\Phi_{s,t}(\mu_1)-\Phi_{s,t}(\mu_2)\right\|_{TV}\leq 2  \prod_{k=0}^{n}\left(1- d_{s_k,t_k}\right).
  \end{align*}
  This remark also applies to the next results (Proposition~\ref{prop:eta} and Theorem~\ref{thm:Q-proc}), where $\nu_s$ and $d'_s$
  can also be modified accordingly.
\end{rem}

Note also that our result is optimal in the time-homogeneous setting, in the sense that the exponential contraction
in~\eqref{eq:main-2} is equivalent to the property $d_0>0$ (see~\cite[Thm.\,2.1]{champagnat-villemonais-15}). We leave the extension
of this result to the general time-inhomogeneous case as an open question.

\section{Convergence of the expected penalization and penalized process up to infinite time}
\label{sec:other-results}

In the absorbed time-homogeneous setting of~\cite{champagnat-villemonais-15}, we also obtained complementary results on the limiting
behavior of $\E_x(Z_{s,t})$ when $t\rightarrow\infty$ (with $Z_{s,t}=\11_{X_t\not\in D}$ as in~\eqref{eq:ex-penaliz}) and on the
penalized process conditioned to never be extinct. Both statements can be extended to the present time-inhomogeneous penalized
framework, as stated in the following two results.

\begin{prop}
\label{prop:eta}
For all $y\in E$ and $s\in I$ such that $d'_s>0$, there exists a finite constant $C_{s,y}$ only depending on $s$ and $y$ such that,
for all $x\in E$ and $t,u\geq s+1$ with $t\leq u$,
\begin{align}
  \label{eq:prop-borne-eta}
  \left|\frac{\E_{s,x}(Z_{s,t})}{\E_{s,y}(Z_{s,t})}-\frac{\E_{s,x}(Z_{s,u})}{\E_{s,y}(Z_{s,u})}\right|\leq
  C_{s,y}\inf_{v\in[s+1, t]}\frac{1}{d'_v}\prod_{k=0}^{\lfloor v-s \rfloor-1}\left(1- d_{v-k}\right).
\end{align}
In particular, if
\begin{align}
\label{eq:hyp-prop-eta}
  \liminf_{t\in I,\ t\rightarrow+\infty}\frac{1}{d'_t}\prod_{k=0}^{\lfloor t-s \rfloor-1}\left(1- d_{t-k}\right)=0,
\end{align}
for all $s\geq 0$, there exists a positive bounded function $\eta_s:E\rightarrow (0,+\infty)$ such that
\begin{align}
\label{eq:limite-eta}
\lim_{t\rightarrow\infty} \frac{\E_{s,x}(Z_{s,t})}{\E_{s,y}(Z_{s,t})}=\frac{\eta_s(x)}{\eta_s(y)},\quad \forall x,y\in E,
\end{align}
where, for any fixed $y$, the convergence holds uniformly in $x$, and such that, for all $x\in E$ and $s\leq t\in I$,
\begin{align}
  \label{eq:fonction-propre}
  \E_{s,x}(Z_{s,t}\eta_t(X_t))=\eta_s(x).
\end{align}
In addition, the function $s\mapsto \|\eta_s\|_\infty$ is locally bounded on $[0,+\infty)$.
\end{prop}

Since $d'_t\leq d_t$, there is convergence to 0 in~\eqref{eq:prop-borne-eta} if $\limsup d'_t>0$, and the convergence is geometric if
$\inf_{t\geq 0}d'_t>0$. There is also convergence to 0 for example if $d'_t\geq c t^{-1}$ for $t$ large enough and for some
$c>1$.

The last theorem also implies uniqueness results on equation~\eqref{eq:fonction-propre} and on associated PDE problems.

\begin{cor}
  \label{cor:unicite-EDP}
  Assume that
  \begin{align*}
    \liminf_{t\in I,\ t\rightarrow+\infty}\frac{1}{d'_t}\prod_{k=0}^{\lfloor t-s \rfloor-1}\left(1- d_{t-k}\right)=0,
  \end{align*}
  Then the function $(s,x)\mapsto\eta_s(x)$ of the last proposition is the unique solution $(s,x)\mapsto f_s(x)$, up to a
  multiplicative constant, of
  \begin{align}
    \label{eq:valeur-propre-sg}
    \E_{s,x}(Z_{s,t}f_t(X_t))=f_s(x)
  \end{align}
  such that $f_s$ is bounded for all $s\geq 0$ and for some $x_0\in E$,
  \begin{align}
    \label{eq:condition}
    \|f_t\|_\infty=o\left(\frac{\prod_{k=0}^{\lfloor t\rfloor-1}\left(1-d_{t-k}\right)^{-1}}{\E_{0,x_0}(Z_{0,t})}\right)
  \end{align}
  when $t\rightarrow+\infty$. Moreover, this unique solution $f_s$ of~\eqref{eq:valeur-propre-sg} can be chosen positive.
\end{cor}

\begin{rem}
  \label{rem:uniqueness-EDP}
  The last result also gives uniqueness properties for stationary time-in\-ho\-mo\-ge\-neous evolution equations with growth
  conditions at infinity. Namely, let us assume that the semigroup $P_{s,t}f(x)=\E_{s,x}[Z_{s,t}f(X_t)]$ admits as time-inhomogeneous
  infinitesimal generator $(L_t,t\geq 0)$ (as defined e.g.\ in~\cite[Ch.\,5]{pazy-83}). We can also define the (time-homogeneous)
  semigroup on $[0,+\infty)\times E$ by $T_tf(s,x)=P_{s,s+t}f(s+t,x)$. Then~\eqref{eq:valeur-propre-sg} writes $T_t\eta=\eta$ and
  hence can be interpreted as some form of weak solution of the evolution equation
  \begin{align}
    \label{eq:PDE}
    \partial_t f_t(x)+L_t f_t(x)=0,\quad\forall (s,x)\in[0,+\infty)\times E,
  \end{align}
  for which Proposition~\ref{prop:eta} and Corollary~\ref{cor:unicite-EDP} give existence and uniqueness under
  condition~\eqref{eq:condition}.
\end{rem}

\begin{thm}
\label{thm:Q-proc}
Assume that
 \begin{align*}
  \liminf_{t\in I,\ t\rightarrow+\infty}\frac{1}{d'_t}\prod_{k=0}^{\lfloor t-s \rfloor-1}\left(1- d_{t-k}\right)=0.
\end{align*}
Then, for all $s\in I$, the family $(\QQ_{s,x})_{s\in I,x\in E}$ of probability measures on $\Omega$ defined by
$$
\QQ_{s,x}(A)=\lim_{T\rightarrow+\infty}\P_{s,x}(A\mid T<\tau_\d),\ \forall A\in{\cal F}_{s,u},\ \forall u\geq s,
$$
is well defined and given by
\begin{align*}
\restriction{\frac{d\Q_{s,x}}{d\P_{s,x}}}{{\cal F}_{s,u}}=\frac{Z_{s,u}\eta_u(X_u)}{\E_{s,x}[Z_{s,u}\eta_u(X_u)]},
\end{align*}
and the process $(\Omega,({\cal F}_{s,t})_{t\geq s},(X_t)_{t\geq
  0},(\QQ_{s,x})_{s,\in I,x\in E})$ is an $E$-valued time-inho\-mo\-ge\-neous Markov process.
In addition, this process is asymptotically mixing in the sense that, for any $s\leq t\in I$ and $x\in E$,
\begin{align}
\label{eq:thm-Q-proc-3}
\left\|\Q_{s,x}(X_t\in\cdot)-\Q_{s,y}(X_t\in\cdot)\right\|_{TV}\leq 2 \prod_{k=0}^{\lfloor t-s \rfloor -1}\left(1-d_{t-k}\right).
\end{align}
\end{thm}

\begin{rem}
  In the case where $Z_{s,u}$ admits a regular conditional probability given $X_u$ for all $s\leq u$ (for example if $E$ is a Polish
  space), the transition kernel of $X$ under $(\Q_{s,x})_{s,x}$ is given by
  \begin{align*}
    \tilde{p}(s,x;u,dy)=\frac{\E_{s,x}(Z_{s,u}\mid X_u=y)\,\eta_u(y)}{\E_{s,x}(Z_{s,u}\eta_u(X_u))} p(s,x;u,dy),
  \end{align*}
  where $p$ is the transition kernel of the process $X$ under $(\P_{s,x})_{s,x}$.
\end{rem}

\section{One-dimensional diffusions with time-dependent coefficients}
\label{sec:exa_diff}

Our first example of application of the results of Section~\ref{sec:intro} deals with the case of a Markov process conditioned not to
hit some absorbing point $\d$, i.e.
\begin{align*}
Z_{s,t}=\11_{t<\tau_\d},
\end{align*}
where $\tau_\d$ is the hitting time of $\d$. This is the setting of~\cite{champagnat-villemonais-15}, but we study here the
time-inhomogeneous case.

More precisely, we consider a time inhomogeneous one-dimensional diffusion process $X$ on $[0,+\infty)$ stopped when it hits $0$ at time
$T^X_0=\inf\{t\geq 0,\ X_{t-}=0\}$ assumed almost surely finite and solution, for all $s\geq 0$, on $[s,T^X_0)$ to
\begin{align}
\label{eq:the-eds}
dX_t=\sigma(t,X_t)dB_t,\quad X_0\in(0,+\infty),
\end{align}
where $B$ is a standard one-dimensional Brownian motion and $\sigma$ is a measurable function on $[0,+\infty)\times(0,+\infty)$ to
$(0,+\infty)$. Note that our result could of course also apply to any time-inhomogeneous diffusions with drift that can be put in the
previous form by a time-dependent change of spatial scale. We assume that
\begin{align*}
\sigma_*(x)\leq \sigma(t,x)\leq \sigma^*(x),
\end{align*}
for some measurable functions $\sigma^*$ and $\sigma_*$ from $(0,+\infty)$ to $[0,+\infty]$ satisfying 
\begin{align*}
\int_{(0,+\infty)}\frac{x\,dx}{\sigma_*(x)^2}<\infty\quad\text{ and }\quad\int_{(a,b)}\frac{dx}{\sigma^*(x)^2}>0,\ \forall 0<a<b<\infty.
\end{align*}
Note that the former condition means that the time-homogeneous diffusion $dY_t=\sigma_*(Y_t)dB_t$ on $(0,\infty)$ stopped when it
hits $0$ at time $T^Y_0$ admits $+\infty$ as entrance boundary (i.e.\ $Y$ comes down from infinity, as defined in~\cite{CCLMMS09})
and  that $T^Y_0<\infty$ almost surely (see e.g.~\cite{freedman-83}).

We also assume that the time-homogeneous diffusion process $Y$ satisfies, for some constants $t_1>0$ and $A>0$,
\begin{align}
\label{eq:hyp-Y}
\P_y(t_1<T^Y_0)\leq A y,\ \forall y>0.
\end{align}
Up to a linear transformation of time (or, equivalently, multiplying $\sigma(t,x)$ by some postive constant), we can---and
will---assume without loss of generality that $t_1<1$. Explicit conditions on $\sigma_*$ ensuring the last assumption are given
in~\cite[Thms\,3.4\,\&\,3.7]{champagnat-villemonais-15b}. For instance, these conditions are fulfilled if $\sigma_*(x)\geq C
x\log^{\frac{1+\varepsilon}{2}}\frac{1}{x}$ for some constants $C>0$ and $\varepsilon>0$ in a neighborhood of $0$. Note that if
$\varepsilon=0$, the condition $\int_{0+}\frac{x\,dx}{\sigma_*(x)^2}<\infty$ might not be satisfied and hence it is not guaranteed
that the diffusion $Y$ hits $0$ in finite time.

\begin{thm}
\label{thm:diff}
Under the above assumptions, 
\begin{align*}
\inf_{s\geq 1} d'_s>0.
\end{align*}
In particular, we obtain exponential convergence in~\eqref{eq:main-1}, \eqref{eq:main-2}. Moreover, the assumptions of Proposition~\ref{prop:eta} and Theorem~\ref{thm:Q-proc} are satisfied.
\end{thm}

As far as we know, this is the first result of this kind on time-inhomogeneous diffusions allowing non-periodic or non-regular or degenerate coefficients. In particular, this extends significantly the results of~\cite{DelMoralVillemonais2015} in the one-dimensional case.

\begin{proof}[Proof of Theorem~\ref{thm:diff}]
The proof follows the same steps as in \cite[Section~5.1]{champagnat-villemonais-15b}, making use of the next lemma.
\begin{lem}
\label{lem:diff}
There exist constants $t_1\in]0,1[$ and $A>0$ such that,  for all $s\geq 0$ and $x>0$, 
\begin{align}
\label{eq:eq1lemdiff}
\P_{s,x}(s+t_1< T^X_0)\leq Ax\quad\text{and}\quad \inf_{s\geq 0}\, \P_{s,x}(s+t<T^X_0)>0,\ \forall t\geq 0.
\end{align}
Moreover, for all $t_2>0$,
\begin{align}
\label{eq:eq2lemdiff}
  \inf_{s\geq 0,x>0}\P_{s,x}(T^X_0 < s+t_2)>0,
\end{align}
for all $\rho>0$, there exists $b_\rho>0$ such that
  \begin{equation}
    \label{eq:moment-expo}
    \sup_{s\geq 0,x\geq b_\rho}\E_{s,x}(e^{\rho(T_{b_\rho}-s)})<+\infty,
  \end{equation}
 for all $a>0$ and $t\geq 0$,
  \begin{equation}
   \label{eq:proba-over-a}
   \inf_{s\geq 0} \,\P_{s,a}(t+s<T^X_{a/2})>0
  \end{equation}
  and for all $a,b>0$, there exists $t_{a,b}>0$ such that for all $t\geq t_{a,b}$,
  \begin{equation}
   \label{eq:irreduct}
   \inf_{s\geq 0}\P_{s,a}(X_{s+t}\geq b)>0.
  \end{equation}
\end{lem}

We admit for the moment this result and extend the main steps of~\cite[Section~5.1]{champagnat-villemonais-15b} to our new setting.

  \medskip
  \noindent
  \textit{Step 1: the conditioned process escapes a neighborhood of 0 in finite time.}\\
  The goal of this step is to prove that there exists $\varepsilon,c>0$ such that
  \begin{equation}
    \label{eq:step1}
    \P_{s,x}(X_{s+t_1}\geq\varepsilon\mid s+t_1<T^X_0)\geq c,\quad\forall s\geq 0,\,x>0.    
  \end{equation}

  To prove this, we first observe that, since $X$ is a local martingale and since $X_{(s+t_1)\wedge T^X_1}=0$ on the event $T_0^X\leq (s+t_1)\wedge T^X_1$, for all $x\in(0,1)$,
\begin{align*}
  x&=\E_{s,x}(X_{(s+t_1)\wedge T^X_1})=\E_{s,x}\left(X_{(s+t_1)\wedge T^X_1}\11_{(s+t_1)\wedge T^X_1<T_0^X}\right)\\
   &=\P_{s,x}(s+t_1<T^X_0)\E_{s,x}(X_{(s+t_1)\wedge T^X_1}\mid s+t_1<T^X_0)+\P_{s,x}(T^X_1<T^X_0\leq s+t_1).
\end{align*}
  By the Markov property,
  \begin{align*}
  \P_{s,x}(T^X_1<T^X_0\leq s+t_1)&\leq \EE_{s,x}\left[\mathbbm{1}_{T^X_1<T^X_0\wedge (s+t_1)}\P_{T^X_1,1}(T^X_0\leq s+t_1)\right]\\
  &\leq \P_{s,x}(T^X_1<T^X_0)  \sup_{u\in[s,s+t_1]}\P_{u,1}(T^X_0\leq s+t_1)\\
  &\leq \P_{s,x}(T^X_1<T^X_0)  \sup_{u\in[s,s+t_1]}\P_{u,1}(T^X_0\leq u+t_1)\\
  &=x \sup_{u\in[s,s+t_1]}\P_{u,1}(T^X_0\leq u+t_1).
  \end{align*}
The second part of Equation~\eqref{eq:eq1lemdiff} of  Lemma~\ref{lem:diff} entails that $\sup_{u\geq 0}\P_{u,1}(T^X_0\leq u+t_1)<1$ and therefore, using the first part of Equation~\eqref{eq:eq1lemdiff} of Lemma~\ref{lem:diff},
  $$
  \E_{s,x}\left(1-X_{(s+t_1)\wedge T^X_1}\mid s+t_1<T^X_0\right)\leq 1-\frac{1}{A'},
  $$
with $A'=A/(1-\sup_{u\geq 0}\P_{u,1}(T^X_0\leq u+t_1))$.
Markov's inequality then implies that, for all $x\in(0,1)$,
\begin{equation}
  \label{eq:calcul2.0}
  \P_{s,x}\left(\left. X_{(s+t_1)\wedge T^X_1}\leq \frac{1}{2A'-1}\,\right|\, s+t_1< T^X_0\right)\leq \frac{1-1/A'}{1-1/(2A'-1)}=1-\frac{1}{2A'}.    
\end{equation}
Set $\varepsilon:=1/(2(2A'-1))$ and assume, without loss of generality, that $A'$ is big enough so that $2\varepsilon\in(0,1)$. Applying the second part of~\eqref{eq:eq1lemdiff} to the diffusion $dZ_t=\sigma_*(t,Z_t+\varepsilon)$ (which satisfies
the above assumptions since $\int_0^\infty \frac{x\,dx}{\sigma_*(\varepsilon+x)^2}\leq\int_\varepsilon^\infty
\frac{x\,dx}{\sigma_*(x)^2}<\infty$), we have
\begin{align*}
  \inf_{t\geq 0}\P_{t,2\varepsilon}\left(t+t_1<T^X_\varepsilon\right)>0.
\end{align*}
Hence, for all $x\in(0,2\varepsilon)$,
\begin{align*}
  \P_{s,x}(X_{s+t_1}\geq\varepsilon) & \geq\P_{s,x}\left(T^X_{2\varepsilon}<s+t_1\right)\inf_{t\geq 0}\P_{t,2\varepsilon}\left(t+t_1<T^X_\varepsilon\right) \\
  & \geq\P_{s,x}\left(X_{(s+t_1)\wedge T^X_1}\geq 2\varepsilon\right) \inf_{t\geq 0}\P_{t,2\varepsilon}\left(t+t_1<T_\varepsilon\right) \\
  & \geq \frac{\P_{s,x}\left(s+t_1<T^X_0\right)}{2A'}\inf_{t\geq 0}\P_{t,2\varepsilon}\left(t+t_1<T_\varepsilon\right)
\end{align*}
by~\eqref{eq:calcul2.0}. This ends the proof of~\eqref{eq:step1} for $x<2\varepsilon$. For $x\geq 2\varepsilon$, standard
coupling arguments entail
$$
\P_{s,x}(X_{t_1}>\varepsilon\mid
t_1<\tau_\d)\geq\P_{s,x}(X_{t_1}>\varepsilon)\geq\P_{s,x}(t_1<T^X_\varepsilon)\geq\P_{s,2\varepsilon}(T^X_\varepsilon>t_1)>0.
$$
Hence~\eqref{eq:step1} is proved.

\medskip
  \noindent
  \textit{Step 2: Construction of coupling measures for the unconditioned process.}\\
Set $t_2=1-t_1>0$.  Our goal is to prove that there exists a constant $c_1>0$ such that, for all $s\geq 0$ and $x\geq\varepsilon$,
  \begin{equation}
    \label{eq:step2}
    \P_{s,x}(X_{s+t_2}\in \cdot)\geq c_1\pi_s(\cdot),
  \end{equation}
  where 
  $$
  \pi_s(\cdot)=\P_{s,\varepsilon}(X_{s+t_2}\in\cdot\mid s+t_2<T^X_0).
  $$
 Fix $s\geq 0$ and
  $x\geq\varepsilon$ and construct two independent diffusions $X^{s,\varepsilon}$ and $X^{s,x}$ solution to~\eqref{eq:the-eds} with initial values at time $s$
given by  $\varepsilon$ and $x$ respectively. Let $\theta=\inf\{t\geq s:X^{s,\varepsilon}_t=X^{s,x}_t\}$. By the strong Markov property, the process
  $$
  Y^{s,x}_t=
  \begin{cases}
    X^{s,x}_t & \text{if }t\in[s,\theta], \\
    X^{s,\varepsilon}_t & \text{if }t>\theta
  \end{cases}
  $$
  has the same law as $X^{s,x}$. Since $\theta\leq T^{s,x}_0:=\inf\{t\geq s: X^{s,x}_t=0\}$, for all $t>s$,
  $\P(\theta<t)\geq\P(T^{s,x}_0<t)$. Using Equation~\eqref{eq:eq2lemdiff} of Lemma~\ref{lem:diff}, we have
  $$
  c'_1:=\inf_{s\geq 0,y>0}\P_{s,y}(T^{s,x}_0 < s+t_2)>0.
  $$
  Hence
  $$
  \P_{s,x}(X_{s+t_2}\in\cdot)=\P(Y^{s,x}_{s+t_2}\in\cdot)\geq\P(X^{s,\varepsilon}_{s+t_2}\in\cdot,\ T^{s,x}_0<s+t_2)\geq c'_1\P_{s,\varepsilon}(X_{s+t_2}\in \cdot).
  $$
  Therefore, \eqref{eq:step2} is proved with $c_1=c'_1\inf_{s\geq 0}\P_{s,\varepsilon}(s+t_2<T^X_0)$, which is positive
  by~\eqref{eq:eq1lemdiff} of Lemma~\ref{lem:diff}.

  \medskip
  \noindent
  \textit{Step 3: Proof that $\nu_s\geq c_1 c \pi_{s-1+t_1}$.}\\
  Recall that $t_1+t_2=1$. Using successively the Markov property, Step 2 and Step 1, we have for all $s\geq 1$ and $x>0$
  \begin{align*}
    \P_{s-1,x}(X_{s-1+t_1+t_2}&\in\cdot\mid s-1+t_1+t_2<T^X_0)  \geq \P_{s-1,x}(X_{s}\in\cdot\mid s-1+t_1<T^X_0) \\ 
    & \geq \int_\varepsilon^\infty \P_{s-1+t_1,y}(X_{s}\in\cdot)\P_{s-1,x}(X_{s-1+t_1}\in dy\mid s-1+t_1<T^X_0) \\
    & \geq c_1\int_\varepsilon^\infty \pi_{s-1+t_1}(\cdot)\P_{s-1,x}(X_{s-1+t_1}\in dy\mid s-1+t_1<T^X_0) \\
    & =c_1\pi_{s-1+t_1}(\cdot)\P_{s-1,x}(X_{s-1+t_1}\geq \varepsilon\mid s-1+t_1<T^X_0) \geq c_1 c \pi_{s-1+t_1}(\cdot).
  \end{align*}
  This entails $\nu_s\geq c_1 c \pi_{s-1+t_1}$, where $\nu_s$ is defined in~\eqref{eq:intro-nu-s}.
  
  \medskip
  \noindent
  \textit{Step 4: Proof that $\inf_{s\geq 1} d'_s>0$.}\\  
We set $a=\varepsilon/2$. Using the definition of $\pi_s$ , we have
\begin{align*}
\pi_s([a,+\infty[)&\geq \P_{s,2a}(T_a^X\geq s+t_2\mid s+t_2< T_0^X)\\
                  &\geq \P_{s,2a}(T_a^X\geq s+t_2).
\end{align*}
Inequality~\eqref{eq:proba-over-a} allows us to conclude that $\inf_{s\geq 1}\nu_s([a,+\infty))>0$. 

We also deduce from~\eqref{eq:irreduct} that, setting $t_3=t_{a,a}$, there exists $\rho>0$ such that
$$
\inf_{s\geq 0}\P_{s,a}(X_{s+t_3}\geq a)\geq e^{-\rho t_3}.
$$
From~\eqref{eq:moment-expo}, one can choose $b> a$ large enough so that
\begin{align*}
  A:=\sup_{s\geq 0, x\geq b}\E_{s,x}\left(e^{\rho (T^X_{b}-s)}\right)<\infty.
\end{align*}
Then, defining $T^X_{[0,b]}$ as the first hitting time of $[0,b]$ by the process $X$ and by $\theta_t$ the shift operator of time
$t$, Markov's property entails
\begin{align}
  \label{eq:moment-expo-2}
  \sup_{s\geq 0, x\geq b}\E_{s,x}\left(e^{\rho (T^X_{[0,b]}\circ\theta_t-s-t)}\right)\leq A,
\end{align}
where, under $\P_{s,x}$, $T^X_{[0,b]}\circ\theta_t$ is the first hitting time of $[0,b]$ after time $s+t$ by the process $X$. Note
that, in particular, $T^X_{[0,b]}\circ\theta_t=s+t$ if $T^X_0\leq s+t$.

Then, setting $t_4=t_{a,b}$, for all $u\geq s+t_4$, defining $k$ as the unique interger such that $s+k t_3+t_4\leq u<s+(k+1)t_3+t_4$,
we have by Markov's property
\begin{align*}
  \P_{s,a}(X_u\geq b) & \geq\P_{s,a}(X_{s+t_3}\geq a,\ X_{s+2t_3}\geq a,\ldots,X_{s+kt_3}\geq a,\ X_u\geq b) \\
  & \geq e^{-\rho kt_3}\inf_{v\geq 0}\P_{v,a}(X_{v+u-s-kt_3}\geq b) \\
  & \geq ce^{-\rho(u-s)}
\end{align*}
where $c>0$ by~\eqref{eq:irreduct}. Therefore, for all $t\geq u\geq s+t_4$, making use of the monotonicity of $x\mapsto\P_{s,x}(t<T^X_0)$,
\begin{align}
  \label{eq:last-de-chez-last}
  ce^{-\rho(u-s)}\P_{u,b}(t<T^X_0)\leq \P_{s,a}(X_u\geq b) \P_{u,b}(t<T^X_0)\leq \P_{s,a}(t<T^X_0).
\end{align}

Then, for all $x\geq b$ and all $t\geq s+t_4$, using successively the strong Markov property, Equation~\eqref{eq:moment-expo-2} with $t=t_4$,~\eqref{eq:last-de-chez-last}
with $u=t$,~\eqref{eq:last-de-chez-last} with $u\geq s+t_4$, and~\eqref{eq:moment-expo-2} again,
\begin{align*}
\P_{s,x}(t<T^{X}_0)&\leq \P_{s,x}(t< T^X_{[0,b]}\circ\theta_{t_4})+\int_{s+t_4}^t
\sup_{y\in[0,b]}\P_{u,y}(t<T^X_0)\,\P_{s,x}(T^X_{[0,b]}\circ\theta_{t_4}\in du)\\
&\leq Ae^{-\rho (t-s-t_4)}+\int_{s+t_4}^t
\P_{u,b}(t<T^X_0)\,\P_{s,x}(T^X_{[0,b]}\circ\theta_{t_4}\in du)\\
&\leq c^{-1}Ae^{\rho t_4}\P_{s,a}(t<T^X_0)+c^{-1}\P_{s,a}(t<T^X_0)\int_{s+t_4}^t e^{\rho(u-s)}\,\P_x(T^X_{[0,b]}\circ\theta_{t_4}\in du)\\
&\leq 2c^{-1}Ae^{\rho t_4}\P_{s,a}(t<T^X_0).
\end{align*}
In the case where $t\in[s,s+t_4]$,
\begin{align*}
  \P_{s,x}(t<T^{X}_0)\leq 1\leq\frac{\P_{s,a}(s+t_4< T^X_0)}{\inf_{s\geq 0}\P_{s,a}(s+t_4< T^X_0)}\leq\frac{\P_{s,a}(t<
    T^X_0)}{\inf_{s\geq 0}\P_{s,a}(s+t_4< T^X_0)}.
\end{align*}
We deduce from inequality~\eqref{eq:eq1lemdiff} of Lemma~\ref{lem:diff} that there exists a constant $C>0$ such that, for all $s\geq 0$
and $t\geq s$,
\begin{align*}
  \sup_{x>0}\P_{s,x}(t<T^X_0)=\sup_{x\geq b}\P_{s,x}(t<T^X_0)\leq C\P_{s,a}(t<T^X_0).
\end{align*}
Since $\inf_{x\geq a}\P_{s,x}(t<T^X_0)=\P_{s,a}(t<T^X_0)$ and $\inf_{s\geq 1}\nu_s([a,+\infty))>0$, we obtain
\begin{align*}
\inf_{s\geq 1} d'_s>0.
\end{align*}
This concludes the proof of Theorem~\ref{thm:diff}.
\end{proof}

\begin{proof}[Proof of Lemma~\ref{lem:diff}]
We assume in the whole proof that $s=0$ and $X_0=x$. Since the statements of Lemma~\ref{lem:diff} are obtained from comparisons
with time-homo\-geneous diffusions, the result will follow from the study of the case $s=0$ only. For all $t\geq 0$, let
\begin{align*}
b(s)=\int_0^s \sigma^2(u,X_u)\,du.
\end{align*}
Note that $b$ is continuous and increasing. The equality $X_t=W_{b(t)}$ for all $t< T^X_0$ defines a Brownian motion $W$ started at
$W_0=x$ and stopped at its first hitting time of $0$ denoted by 
$T_0^W=b(T_0^X)$. 
This is a classical consequence of Levy's characterisation of the Brownian motion, see for
instance~\cite{revuz-yor-91a}. Note that, since a one dimensional Brownian motion hits $0$ in finite time almost surely, there exists $t\geq 0$ such that $W_{b(t)}=0$ and hence $b(T_0^X)<\infty$ almost surely.

 Let $Y$ be the time-homogeneous diffusion process stopped at $0$ defined as $Y_t=W_{b_*(t)}$, where
\begin{align*}
b_*(t)=\inf\{s\geq 0,\ a_*(s)\geq t\}, \quad\text{ with }a_*(s)=\int_0^s \frac{du}{\sigma_*(W_u)^2}.
\end{align*}
And similarly for $Z_t=W_{b^*(t)}$, replacing $\sigma_*$ by $\sigma^*$. In particular, $Y_{a_*(t)}=W_t$, $Z_{a^*(t)}=W_t$, and hence $T^Y_0=a_*(T^W_0)$ and $T^Z_0=a^*(T^W_0)$.
Note that $Y$ and $Z$ are solutions of the time-homogeneous SDEs
\begin{align*}
dY_t=\sigma_*(Y_t) dB^Y_t\quad\text{and}\quad dZ_t=\sigma^*(Z_t) dB^Z_t,\quad\text{with }Y_0=Z_0=x,
\end{align*}
for some Brownian motions $B^Y$ and $B^Z$ with $Y_0=Z_0=x$. The interest of this construction is that the processes $Y$ and $Z$ are
both obtained from a random time change of $X$: $Y$ is obtained by a slowing down of $X$, and $Z$ by a speeding up of $X$. In particular, it
is easy to check that
\begin{align*}
a_*(b(t))\geq t\quad\text{ and }\quad a^*(b(t))\leq t,\ \forall t\geq 0.
\end{align*}
Hence $T_0^Z\leq T_0^X\leq T_0^Y$ almost surely. Therefore, the first inequality of~\eqref{eq:eq1lemdiff} follows from the same
property for $Y$, as assumed in~\eqref{eq:hyp-Y}. Similarly, the second inequality in~\eqref{eq:eq1lemdiff} and~\eqref{eq:eq2lemdiff}
follow from the same property for $Z$ and $Y$, respectively, which are standard properties of time-homogeneous diffusion processes
(see for instance~\cite{freedman-83}).

Using the previous argument, we also deduce that, $T_a^Z\leq T_a^X\leq T_a^Y$ almost surely for all $a\leq x$.
Hence~\eqref{eq:moment-expo} follows from the same property for $Y$, which is classical because infinity is an entrance boundary for
$Y$ (see for instance~\cite{CCLMMS09,champagnat-villemonais-15b}). Inequality~\eqref{eq:proba-over-a} also follows from the same
comparison of hitting times and standard regularity properties of the time-homogeneous diffusion $Z$.

Finally, if $b\leq a/2$~\eqref{eq:irreduct} follows directly from~\eqref{eq:proba-over-a}, and if $b>a/2$, we use the comparison with
$Y$ and the fact that $\P_a(T^Y_{2b}<t_0)>0$ for some $t_0>0$ to see that $X$ hits $2b$ before time $t_0$ with probability under
$\P_{s,a}$ uniformly bounded from below with respect to $s\geq 0$. Next we use the comparison with $Z$ (as we did to prove the second
inequality in~\eqref{eq:eq1lemdiff}) to see that, under $\P_{s,2b}$, for any $t\geq 0$, there is a uniformly (with respect to $s$) positive
probability that $X$ does not hit $a<2b$ before time $t$. Combining these two facts entails~\eqref{eq:irreduct} with $t_{a,b}=t_0$.
\end{proof}

\section{Penalized time-inhomogeneous birth and death pro\-ces\-ses}
\label{sec:exa_PNM}

In the previous example, we considered the case of an inhomogeneous Markov process which is uniformly dominated by a time homogeneous
process coming down from infinity. This provided uniform mixing, controled by the Dobrushin coefficient, given by the mass of the
measure $\nu_s$. The goal of this section is to study a case of inhomogeneous Markov process in continuous time alternating periods
of uniform mixing (i.e.\ uniform coming down from infinity) and periods without uniform mixing.

This situation is for example natural for a birth and death process in random environment, where the environment alternates periods
favorable to growth and periods where the population has a tendency to descrease. The study of quasi-stationary behavior of such a
population can be formulated in two different ways: the study of convergence of the distribution of the population conditional on
non-extinction 1) when expectations are taken with respect to the law of the environment and of the birth and
death process (so-called \emph{annealed} quasi-stationary behavior), and 2) when expectations are taken only with respect to the
law of the birth and death process, for any fixed realization of the environment (so-called \emph{quenched} quasi-stationary
behavior). In the case of time-homogeneous Markov environment dynamics, the joint dynamics of environment and population is
time-homogeneous and hence enters the scope of our general results for Markov processes of~\cite{champagnat-villemonais-15}. The case
of quenched quasi-stationary behavior is more delicate since all realizations of the environment must be considered, even those which
are very unlikely when the population is conditioned on survival. In particular, this requires more stringent irreducibility
assumptions (see~\eqref{eq:assumption3} and~\eqref{eq:assumption4} below) than what one would expect in the annealed case.

We also detail in the examples studied in this section how inhomogeneous penalization can be handled, with some appropriate
boundedness assumptions. Our method can actually be adapted to several Markov processes with similar penalization. Typical situations
include the models studied in~\cite{champagnat-villemonais-15,champagnat-villemonais-15b,champagnat-villemonais-15c,
  champagnat-villemonais-15d}.

\subsection{General result}
\label{sec:PNM-general}

Let $(X_t)_{t\in\R_+}$ be a time inhomogeneous birth and death process reflected at $1$, with measurable birth rates $b_i(t)>0$ and
death rates $d_i(t)\geq 0$ at time $t\geq 0$ from state $i\geq 1$, such that $d_1(t)=0$ for all $t\geq 0$ and $d_i(t)>0$ for $i\geq
2$. We also consider the penalization defined by
\begin{align*}
Z_{s,t}=e^{\int_s^t \kappa(u,X_u) du},
\end{align*}
where $\kappa:\R_+\times \{1,2,\ldots\}\rightarrow \R$ is a bounded measurable function. Note that the study of the distribution of a
birth and death process $Y$ on $\Z_+$ absorbed at $0$ (with the same coefficients except $d_1(t)>0$) and conditioned not to hit $0$
(i.e. penalized by $\11_{Y_t\neq 0}$) enters this setting since
\begin{align*}
\E_{x,s}(f(Y_t)\mid Y_t\neq 0)=\frac{\E_{x,s}\left(f(X_t) e^{-\int_s^t d_1(u)\11_{X_u=1} du}\right)}{\E_{x,s}\left(e^{-\int_s^t d_1(u)\11_{X_u=1} du}\right)}.
\end{align*}
Similarly, the case of birth and death processes with catastrophe (i.e.\ with killing) occurring at bounded rate depending on the
position of the process (see~\cite[Section\,4.1]{champagnat-villemonais-15}) also enters this setting.

We will need irreducibility and stability assumptions: 
\begin{align}
  \label{eq:assumption3}
  \gamma_F:=\inf_{s\geq 0,\ x,y\in F}\P_{s,x}(X_{s+1}=y)>0,\quad\text{for all finite }F\subset\N
\end{align}
and
\begin{align}
  \label{eq:assumption4}
  \rho_x:=\inf_{s\geq 0,\ u\in[s,s+1]}\P_{s,x}(X_u=x)>0,\quad\forall x\in\N.
\end{align}
These two conditions are satisfied for example if, for each $n\in\N$, the functions $b_n(t)$ and $d_n(t)$ are uniformly bounded and
bounded away from 0.

\begin{thm}
  \label{thm:bd-beau-thm}
  Assume that~\eqref{eq:assumption3} and~\eqref{eq:assumption4} hold true and that, for some
  $\lambda>\|\kappa\|_\infty+\log(\gamma_{\{1\}}^{-1})$, there exists a finite $F\subset\N$ and an unbounded $\mathcal{T}\subset\R_+$
  such that
  \begin{align}
    \label{eq:expo-moment-X}
    A:=\sup_{t\in\mathcal{T}}\ \sup_{x\in\N}\ \E_{t,x}\left(e^{\lambda (T^X_F-t)}\right)<\infty,
  \end{align}
  where $T^X_F$ is the first hitting time of the set $F$ by $X$. We also assume that there exists $b\geq 2$ such that the set
  \begin{align}
    \label{eq:expo-moment-2}
    \mathcal{T}_b:=\Big\{s_1\in\mathcal{T},\,\exists s_2\in\mathcal{T}\text{ s.t.}\quad t_0+2\leq
      s_2-s_1\leq t_0+b\Big\}
  \end{align}
  is unbounded, where
  \begin{align}
    \label{eq:def-t_0}
    t_0=\left\lceil\frac{\log A}{\log(\gamma_{\{1\}}^{-1})}\right\rceil.
  \end{align}
  Then, there exist $\gamma>0$ such that, for all probability measures $\mu_1,\mu_2$ on $\N$ and for all $s\in\N$ and $t\geq s$,
  \begin{align*}
    \left\|\mu_1 K_{s,t}^T-\mu_2 K_{s,t}^T\right\|_{TV}\leq \exp\left(-\gamma\ N_{b,s,t}\right)\ \|\mu_1-\mu_2\|_{TV}
  \end{align*}
 and
  \begin{align*}
    \left\|\Phi_{s,t}(\mu_1)-\Phi_{s,t}(\mu_2)\right\|_{TV}\leq 2 \exp\left(-\gamma\ N_{b,s,t}\right),
  \end{align*}
  where $N_{b,s,t}:=\textnormal{Card}\left\{k\in\N\cap[s,t-t_0-2]:\mathcal{T}_b\cap[k,k+1)\neq\emptyset\right\}$.
  Moreover, the conclusions of Proposition~\ref{prop:eta} and Theorem~\ref{thm:Q-proc} are satisfied, except
  for~\eqref{eq:prop-borne-eta} and~\eqref{eq:thm-Q-proc-3}, which have to be replaced respectively by
  \begin{align*}
    \left|\frac{\E_{s,x}(Z_{s,t})}{\E_{s,y}(Z_{s,t})}-\frac{\E_{s,x}(Z_{s,u})}{\E_{s,y}(Z_{s,u})}\right|\leq
    C_{s,y} \exp\left(-\gamma\ N_{b,s,t}\right),\quad\forall x,y\in E,\ \forall s\leq t\leq u,
  \end{align*}
  for some constant $C_{s,y}$ only depending on $s$ and $y$, and
  \begin{align*}
    \left\|\Q_{s,x}(X_t\in\cdot)-\Q_{s,y}(X_t\in\cdot)\right\|_{TV}\leq 2 \exp\left(-\gamma\ N_{b,s,t}\right),\quad\forall x,y\in\N.
  \end{align*}
\end{thm}

Since $\mathcal{T}_b$ is unbounded, we obtain in particular convergence in total variation in Theorem~\ref{thm:bd-beau-thm}.
Moreover, the exponential speed of convergence is governed by the asymptotic density of the set $\mathcal{T}_b$. In Subsection~\ref{sec:PNM-ex}, we apply Theorem~\ref{thm:bd-beau-thm} to the case of a birth and death process evolving in a quenched random environment.

\begin{proof}
  We first notice that replacing $\kappa$ by $\kappa-\|\kappa\|_\infty$ does not change the operators $\Phi$ and $K$
  in~\eqref{eq:def-phi} and~\eqref{eq:def-K}, and hence the measures $\nu_s$ and the constants $d'_s$ are not modified. Therefore, we
  can assume without loss of generality that $\kappa$ is non-positive. As observed before Theorem~\ref{thm:main-result}, the
  penalized process can then be interpreted as a time-inhomogeneous birth and death process $Y$ with killing. More precisely, let $Y$
  be the time inhomogeneous birth and death process on $\Z_+$ with birth and death rates $b_n(t)$ and $d_n(t)$ at time $t$ from state
  $n\geq 1$, with additional jump rate $-\kappa(t,n)$ at time $t$ from $n\geq 1$ to $0$, which is assumed to be an absorbing point.
  Then
  $$
  \Phi_{s,t}(\mu)(f)=\EE_{\mu,s}(f(Y_t)\mid Y_t\neq 0).
  $$
  
  The process $Y$ can be constructed from the paths of $X$ with an additional killing rate, in which case $T_F\wedge T_0\leq T^X_F$,
  where $T_F$ is the first hitting time of the set $F$ by $Y$, and $T_0=T_{\{0\}}$. Therefore, assumption~\eqref{eq:expo-moment-X}
  implies that, for some constant $A<\infty$, for all $s\in\mathcal{T}$,
  \begin{align}
    \label{eq:expo-moment-Y}
    \sup_{x\in\N}\E_{s,x}\left(e^{\lambda (T_F\wedge T_0-s)}\right)\leq A.
  \end{align}

  \medskip\noindent{\it Step 1: Preliminary computations.}

  Let $s<s+1\leq t$ and $u\in[s+1,t]$. For all $x\in F$, by Markov's property,~\eqref{eq:assumption3} and~\eqref{eq:assumption4},
  \begin{multline*}
    \left(e^{-\|\kappa\|_\infty}\gamma_{\{1\}}\right)^{\lfloor u-s\rfloor-1}\,
    e^{-\|\kappa\|_\infty}\rho_1\,
    e^{-\|\kappa\|_\infty}\gamma_F\,\P_{u,x}(t<T_0) \\
    \begin{aligned}
      &  \leq\P_{s,1}\left(Y_s=Y_{s+1}=\cdots=Y_{s+\lfloor u-s\rfloor-1}=Y_{u-1}=1\right)\,
      \P_{u-1,1}(Y_u=x)\,\P_{u,x}(t<T_0) \\
      & \leq\P_{s,1}(t<T_0).
    \end{aligned}
  \end{multline*}
  Thus, for $C=e^{\|\kappa\|_\infty}\gamma_{\{1\}}/(\rho_1\gamma_F)$ and for all $u\in[s+1,t]$,
  \begin{align}
    \label{eq:calcul}
    e^{-\lambda(u-s)}\sup_{x\in F}\P_{u,x}(t<T_0)\leq C\P_{s,1}(t<T_0).
  \end{align}
  Now, for $u\in[s,s+1]$, by~\eqref{eq:assumption4},
  \begin{align*}
    e^{-\|\kappa\|_\infty}\rho_x\,\P_{u,x}(t<T_0)\leq \P_{s,x}(t<T_0),
  \end{align*}
  and hence, increasing $C$ if necessary, we obtain that for all $u\in[s,t]$,
  \begin{align}
    \label{eq:calcul2}
    e^{-\lambda(u-s)}\sup_{x\in F}\P_{u,x}(t<T_0)\leq C\sup_{x\in F}\P_{s,x}(t<T_0).
  \end{align}

  \medskip\noindent{\it Step 2: Dobrushin coefficient.}

  For this step and the next one, we fix $s_1\in\mathcal{T}_b$ and let $s_2\in\mathcal{T}$ such that $t_0+2\leq s_2-s_1\leq t_0+b$.
  Using~\eqref{eq:expo-moment-Y}, for all $t\geq s_1$ and $x\in\N$,
  \begin{align*}
    \P_{s_1,x}(T_F<t)&= \P_{s_1,x}(T_F<t\wedge T_0)\geq \P_{s_1,x}(t<T_0)-\P_{s_1,x}(t<T_F\wedge T_0)\\
    &\geq e^{-\|\kappa\|_\infty (t-s_1)}-Ae^{-\lambda (t-s_1)}.
  \end{align*}
  Hence, it follows from the definition of $t_0$ in~\eqref{eq:def-t_0} that there exists a constant $c_0>0$ such that
  $\P_{s_1,x}(T_F<t)\geq c_0>0$ for all $t\geq s_1+t_0$.

  By assumption~\eqref{eq:assumption3}, $\inf_{s\geq 0,\ y\in F}\P_{s,y}(Y_{s+1}=1)\geq\gamma_F>0$, thus the Markov property entails
  \begin{align*}
    \P_{s_1,x}(Y_{s_1+t_0+1}=1)
    \geq\mathbb{E}_{s_1,x}\left[\11_{T_F<s_1+t_0}\inf_{u\geq 0,\ y\in F}\mathbb{P}_{u,y}(Y_{u+1}=1)\rho_1^{\lceil t_0\rceil}
      e^{-\|\kappa\|_\infty t_0}\right]\geq c_1,
  \end{align*}
  where the constant $c_1$ does not depend on $s_1\in\mathcal{T}_b$ and $x\in\N$.
  Since for all $x\in\N$ and $f:\N\rightarrow\RR_+$,
  \begin{align*}
    \Phi_{s_1,s_1+t_0+1}(\delta_{x})(f)\geq \E_{s_1,x}[f(Y_{s_1+t_0+1})\11_{s_1+t_0+1<T_0}]\geq f(1)\P_{s_1,x}(Y_{s_1+t_0+1}=1),
  \end{align*}
  we deduce that
  \begin{align*}
    \nu_{s_1,s_1+t_0+1}:=\min_{x\in \N}\Phi_{s_1,s_1+t_0+1}(\delta_{x})\geq c_1\delta_1.
  \end{align*}

  \medskip\noindent{\it Step 3: Comparison of survival probabilities.}

  Given any $s\in\mathcal{T}$, using~\eqref{eq:expo-moment-Y}, Markov's property and inequality~\eqref{eq:calcul2} twice (first with
  $u=t$ and second for all $u\in[s,t]$), we have for all $t\geq s$ and $x\in\N$,
  \begin{align}
    \P_{s,x}(t<T_0) & \leq\P_{s,x}\left(t<T_F\wedge T_0\right)+\P_x\left(T_F\wedge T_0\leq t< T_0\right) \notag \\
    & \leq A e^{-\lambda(t-s)}+\int_s^t\sup_{y\in F\cup\{0\}}\P_{u,y}(t<T_0)\P_{s,x}(T_F\wedge T_0\in du) \notag \\
    & \leq AC\,\sup_{y\in F}\P_{s,y}(t<T_0)+C\,\sup_{y\in F}\P_{s,y}(t<T_0)\int_s^t e^{\lambda (u-s)}\,\P_{s,x}(T_F\wedge T_0\in du)
    \notag \\
    & \leq 2AC\,\sup_{y\in F}\P_{s,y}(t<T_0). \label{eq:calcul3}
  \end{align}

  Recall that we fixed $s_1\in\mathcal{T}_b$ and $s_2\in\mathcal{T}$ such that $t_0+2\leq s_2-s_1\leq t_0+b$. For all $x\in\N$, if $t\geq
  s_2$,~\eqref{eq:calcul3} and~\eqref{eq:calcul} entail
  \begin{align*}
    \P_{s_1+t_0+1,x}(t<T_0) & =\sum_{y\in\N}\P_{s_1+t_0+1,x}(Y_{s_2}=y)\P_{s_2,y}(t<T_0) \\
    & \leq 2AC\sum_{y\in\N}\P_{s_1+t_0+1,x}(Y_{s_2}=y)\sup_{z\in F}\P_{s_2,z}(t<T_0) \\
    & \leq 2AC \sup_{z\in F}\P_{s_2,z}(t<T_0) \\
    & \leq 2AC^2\,e^{\lambda(s_2-(s_1+t_0+1))}\,\P_{s_1+t_0+1,1}(t<T_0) \\
    & \leq 2AC^2\,e^{\lambda(b-1)}\,\P_{s_1+t_0+1,1}(t<T_0).
  \end{align*}
  Since we assumed that the catastrophe rate $-\kappa$ is uniformly bounded, the last inequality extends to any $t\in[s_1+t_0+1,s_2]$
  (increasing the constant if necessary).

  \medskip\noindent{\it Step 4: Conclusion}

  Combining Steps 2 and 3, there exists $c'>0$ such that, for all $s_1\in\mathcal{T}_b$,
  \begin{align*}
    d'_{s_1,s_1+t_0+1} & :=\inf_{t\geq s_1+t_0+1} \frac{\P_{s_1+t_0+1,\nu_{s_1,s_1+t_0+1}}(t<T_0)}{\sup_{x\in \N}
      \P_{s_1+t_0+1,x}(t<T_0)} \\ & \geq c_1 \inf_{t\geq s_1+t_0+1}
    \frac{\P_{s_1+t_0+1,1}(t<T_0)}{\sup_{x\in \N} \P_{s_1+t_0+1,x}(t<T_0)}\geq c'.
  \end{align*}
  Theorem~\ref{thm:main-result} and Remark~\ref{rem:extension} then imply that there exists $\gamma_0>0$ such that
  \begin{align*}
    \left\|\mu_1 K_{s,t}^T-\mu_2 K_{s,t}^T\right\|_{TV}\leq \exp\left(-\gamma_0\ C_{b,s,t}\right)\ \|\mu_1-\mu_2\|_{TV}
  \end{align*}
  and
  \begin{align*}
    \left\|\Phi_{s,t}(\mu_1)-\Phi_{s,t}(\mu_2)\right\|_{TV}\leq 2 \exp\left(-\gamma_0\ C_{b,s,t}\right),
  \end{align*}
  where
  \begin{multline*}
    C_{b,s,t}:=\sup\Big\{k\geq 1:\,\exists s\leq t_1<t_2<\ldots<t_k\leq t-t_0-1,\,t_i\in\mathcal{T}_b,\,\forall i=1,2,\ldots,k, \\
    t_{i+1}-t_{i}\geq t_0+1,\,\forall i=1,2,\ldots,k-1\Big\}.
  \end{multline*}
  Since $N_{b,s,t}\leq (t_0+1) C_{b,s,t}$, this concludes the proof of~\ref{thm:bd-beau-thm} with $\gamma=\gamma_0/(t_0+1)$.
\end{proof}

\subsection{An example with alternating favorable and unfavorable periods in a quenched random environment}
\label{sec:PNM-ex}

To illustrate how the assumptions of Theorem~\ref{thm:bd-beau-thm} can be checked in practice, we consider the case of alternating
phases of favorable and unfavorable birth and death rates. By favorable, we mean a process which comes down fast from infinity (see
Assumption~\eqref{eq:assumption1} below), a criterion which is known to be  related to uniform convergence to quasi-stationary
distributions for time-homogeneous birth and death processes~\cite{Martinez-Martin-Villemonais2012,champagnat-villemonais-15}. We
study the problem of \emph{quenched} stationary behavior of the birth and death process: we assume that the time length of the
favorable and unfavorable periods are the realizations of a random environment and we study properties that hold almost surely with
respect to the environment.

More precisely, we consider two sequences $(u_j,j\geq 0)$ and $(v_j,j\geq 0)$ of positive real numbers and a family of sequence of
pairs of nonnegative real numbers $\{(b_n^j,d_n^j)_{n\geq 1},j\geq 0\}$ such that, for all $j\geq 0$, $d_1^j=0$, $b_n^j>0$ for all
$n\geq 1$ and $d_n^j>0$ for all $n\geq 2$. The sequence $(u_j,j\geq 0)$ (resp.\ $(v_j,j\geq 0)$) represents the lengths of successive
unfavorable (resp.\ favorable) time intervals. Without loss of generality, we assume that the first phase is unfavorable. Therefore,
if we set $s_0=0$
\begin{align*}
  \sigma_j=s_j+u_j\quad\text{and}\quad s_{j+1}=\sigma_j+v_j,\quad\forall j\geq 0,
\end{align*}
then the unfavorable time intervals are $[s_j,\sigma_j)$, $j\geq 0$ and the favorable time intervals are $[\sigma_j,s_{j+1})$, $j\geq 0$.
During each favorable time interval, we assume that the birth and death rates satisfy
\begin{align*}
  b_n(t)\leq b^j_n\quad\text{and}\quad d_n(t)\geq d^j_n,\quad\forall t\in[\sigma_j,s_{j+1}).
\end{align*}

The fact that the process comes down from infinity during favorable time intervals is expressed in the following condition, assumed
throughout this section:
\begin{align}
\label{eq:assumption1}
\sup_{j\geq 0} S_n^j\xrightarrow[n\rightarrow +\infty]{} 0,
\end{align}
where
\begin{align}
\label{eq:def-S}
S_n^j:=\sum_{m\geq n}\frac{1}{d_m^j\alpha^j_m}\sum_{\ell\geq m} \alpha^j_\ell <\infty,
\end{align}
with $ \alpha^j_\ell=\left(\prod_{i=1}^{\ell-1} b^j_i\right)/\left(\prod_{i=1}^{\ell} d_i^j\right). $ For example, easy computations
allow to check that~\eqref{eq:assumption1} is true if, for all $j\geq 0$, $d^j_n\geq a_1(n-1)^{1+\delta}$ and $b_n^j\leq a_2 n$ for
some $a_1,\delta>0$ and $a_2<\infty$.

We recall that, if $S_1^j$ is finite for some $j$, then the time-homogeneous birth and death process $Y^j$ with birth rates $b_i^j$ and death rates $d_i^j$ from state $i$, comes down from infinity (see for instance \cite{vanDoorn1991}). In addition, the distribution of $Y^j$ starting from $\infty$ can be defined and, for all $n\geq 1$,
\begin{align}
\label{eq:S=expect}
S_n^j=\E_{\infty}(T^j_n)=\sum_{\ell\geq n} \E_{\ell+1}(T^j_\ell),
\end{align}
where $T^j_\ell$ is the first hitting time of $i$ by the process $Y^j$. 

In particular, Assumption~\eqref{eq:assumption1} means that on each time interval $[\sigma_j,s_{j+1})$ with $j\geq 0$, the process $X$
comes down from infinity. Note that we make no assumption on the unfavorable time intervals, except that the process is not
explosive.

\begin{rem}
  \label{rem:explosive}
  We could actually deal with explosive processes by defining our process on $\N\cup\{+\infty\}$, assuming that $+\infty$ is absorbing
  during unfavorable time intervals. This would not change our analysis, but for the construction of the process.
\end{rem}

If we think of the time lengths $u_j$ and $v_j$ as modelling the influence of a random environment on the previous birth and death process,
the next result shows that the conditions of Theorem~\ref{thm:bd-beau-thm} are almost surely true for quenched random environments
under very general conditions.

\begin{thm}
  \label{thm:expo-moment-quenched}
  Assume that the times $(u_j,v_j)$ are drawn as i.i.d.\ realizations of a random couple $(U,V)$, where $U$ and $V$ are positive and
  $\E(U)<\infty$. Then, for any $\lambda>0$, there exists a finite $F\subset\N$ and an infinite $J\subset\N$ such that, for almost
  all realization of the random variables $(u_j,v_j)_{j\geq 0}$,
  \begin{align}
    \label{eq:expo-moment-X-bis}
    A_\lambda:=\sup_{j\in J}\ \sup_{x\in\N}\ \E_{\sigma_j,x}\left(e^{\lambda (T^X_F-\sigma_j)}\right)<\infty
  \end{align}
  and for all $t_0>0$, there exists $b\geq 2$ such that the set
  \begin{align}
    \label{eq:expo-moment-2-bis}
    J_b:=\Big\{j\in J,\,\exists k\in J\text{ s.t.}\ t_0+2\leq \sigma_k-\sigma_j\leq t_0+b\Big\}
  \end{align}
  is infinite. If in addition Assumptions~\eqref{eq:assumption3} and~\eqref{eq:assumption4} are satisfied, then the conclusions of
  Theorem~\ref{thm:bd-beau-thm} hold true for almost all realization of the random variables $(u_j,v_j)_{j\geq 0}$.
\end{thm}

\begin{rem}
  \label{rem:cv-expo}
  Note that, since the random variables $(u_j,v_j)$ are i.i.d.\ and because of the renewal argument of the proof of Lemma~\ref{lem:2}
  below, one can check that the set $J_b$ has a positive asymptotic density, in the sense that, for almost all realization of the
  random variables $(u_j,v_j)_{j\geq 0}$,
  \begin{align*}
    \liminf_{T\rightarrow+\infty}\frac{1}{T}\text{Card}\{\sigma_j\leq T: j\in J_b\}>0.
  \end{align*}
  Therefore, under the assumptions of the last theorem, all the convergences in Theorem~\ref{thm:bd-beau-thm} are exponential. More
  precisely, $\exp(-\gamma\ N_{b,s,t})$ can be replaced everywhere in Theorem~\ref{thm:bd-beau-thm} by $C\exp(-\gamma'(t-s))$ for
  some constants $C,\gamma'>0$ \textit{a priori} dependent on the realization of $(u_j,v_j)_{j\geq 0}$.
\end{rem}

Theorem~\ref{thm:expo-moment-quenched} actually holds true under the following more general assumptions. We will divide the proof in
two steps, first proving this more general result (Lemma~\ref{lem:1}) and second, checking that its assumptions are implied by those
of Theorem~\ref{thm:expo-moment-quenched} (Lemma~\ref{lem:2}).

Given fixed positive numbers $u_0,u_1,\ldots$ and $v_0,v_1,\ldots$, we set for all $j\geq 0$ and $\lambda>0$
\begin{align}
\label{eq:assumption2}
C_{\lambda,j}=\sup_{n\in\N}\,  \frac{1}{n}\sum_{\ell=1}^n
\left(  \lambda u_{j+\ell} \, - \, \frac{\log v_{j+\ell-1}}{2} \right).
\end{align}
We will need the next two assumptions: there exists $\lambda>0$ such that
\begin{align}
  \label{eq:assumtion5}
  \exists J_\lambda\subset\N\text{ infinite such that }(C_{\lambda,j},j\in J_\lambda)\text{ is bounded}
\end{align}
and
\begin{align}
  \label{eq:assumtion6}
  \forall t_0>0,\quad\liminf_{j\in J_\lambda,\ j\rightarrow+\infty}\ \inf\left\{\sigma_{k}-\sigma_{j}:\ k\in J_\lambda,\ k>j,\
    \sigma_{k}-\sigma_{j}>t_0\right\}<\infty.
\end{align}

\begin{lem}
  \label{lem:1}
  Assume that there exists $\lambda>0$ such that~\eqref{eq:assumtion5} is satisfied. Then there exists a
  finite $F\subset\N$ such that
  \begin{align}
    \label{eq:expo-moment}
    \sup_{j\in J_\lambda}\ \sup_{x\in\N}\ \E_{\sigma_j,x}\left(e^{\lambda (T_F-\sigma_j)}\right)<\infty.
  \end{align}
  If in addition~\eqref{eq:assumtion6} is satisfied for the same $\lambda>0$, then Conditions~\eqref{eq:expo-moment-X}
  and~\eqref{eq:expo-moment-2} of Theorem~\ref{thm:bd-beau-thm} are true for this value of $\lambda$. In particular, if
  Assumptions~\eqref{eq:assumption3} and~\eqref{eq:assumption4} are satisfied and
  $\lambda>\|\kappa\|_\infty+\log(\gamma_{\{1\}}^{-1})$, then the conclusions of Theorem~\ref{thm:bd-beau-thm} hold true.
\end{lem}

The next lemma shows that the conditions of Lemma~\ref{lem:1} are satisfied almost surely under the conditions of
Theorem~\ref{thm:expo-moment-quenched}. In particular, Theorem~\ref{thm:expo-moment-quenched} is a straightforward consequence of
Lemmata~\ref{lem:1} and~\ref{lem:2}.

\begin{lem}
  \label{lem:2}
  Assume that the times $(u_j,v_j)$ are drawn as i.i.d.\ realizations of a random variable $(U,V)$, where $U$ and $V$ are positive
  and $\E(U)<\infty$. Then, for all $\lambda>0$,~\eqref{eq:assumtion5} and~\eqref{eq:assumtion6} are satisfied.
\end{lem}

\begin{rem}
  \label{rem:other-cases}
  The conditions of Lemma~\ref{lem:1} can be checked in different situations. For example, if for all $j\geq 0$, $v_j\geq\varepsilon$
  for some $\varepsilon>0$, then
  \begin{align*}
    C_{\lambda,j}\leq -\frac{\log\varepsilon}{2}+\lambda\sup_{n\geq 1}\frac{1}{n}\sum_{\ell=1}^n u_{j+\ell}.
  \end{align*}
  As a consequence~\eqref{eq:assumtion5} holds true for any sequence $(u_j,j\geq 0)$ (not necessarily drawn as an independent sequence)
  such that
  \begin{align*}
    \liminf_{j\rightarrow+\infty}\sup_{n\geq 1}\frac{1}{n}\sum_{\ell=1}^n u_{j+\ell}<\infty.
  \end{align*}
\end{rem}

\begin{proof}[Proof of Lemma~\ref{lem:1}]
  For all $s,t\geq 0$, we define
  \begin{align*}
    \alpha(s,t)=\sup_{x\in\N}\ \E_{s,x}\left(e^{\lambda(T_F-s)\wedge t}\right).
  \end{align*}
  For all $j\geq 0$, we have
  \begin{align}
    \label{eq:youpi1}
    \alpha(s_j,t)\leq e^{\lambda u_j}\alpha(\sigma_{j},t)
  \end{align}
  and, by Markov's property,
  \begin{align*}
    \alpha(\sigma_j,t) & \leq\sup_{x\in\N}\ \E_{\sigma_j,x}\left(e^{\lambda(T_F-\sigma_j)\wedge t}\11_{T_F\leq
        s_{j+1}}\right)+\sup_{x\in\N}\ \E_{\sigma_j,x}\left(e^{\lambda v_j}\11_{T_F>s_{j+1}}\right)\alpha(s_{j+1},t) \\
    & \leq\sup_{x\in\N}\ \E_{x}\left(e^{\lambda T^j_F}\right)+\sup_{x\in\N}\ \E_{x}\left(e^{\lambda T^j_F}\11_{T^j_F> v_j}\right)\alpha(s_{j+1},t),
  \end{align*}
  where $T^j_F$ is the first hitting time of the set $F$ by the time homogeneous process $Y^j$ defined above~\eqref{eq:S=expect}.
  Using Cauchy-Schwartz and Markov's inequalities,
  \begin{align*}
    \alpha(\sigma_j,t) & \leq\sup_{x\in\N}\ \E_{x}\left(e^{\lambda T^j_F}\right)+\sup_{x\in\N}\left(\E_{x}\left(e^{2\lambda
          T^j_F}\right)\,\P_x(T^j_F> v_j)\right)^{1/2}\alpha(s_{j+1},t) \\
    & \leq\sup_{x\in\N}\ \E_{x}\left(e^{\lambda T^j_F}\right)+\sup_{x\in\N}\left(\E_{x}\left(e^{2\lambda
          T^j_F}\right)\,\E_x(T^j_F)\right)^{1/2}\frac{\alpha(s_{j+1},t)}{\sqrt{v_j}}
  \end{align*}
  It is standard (cf.\ e.g.~\cite{champagnat-villemonais-15}) to deduce from~\eqref{eq:assumption1} that, given $\lambda>0$, there
  exists a finite $F_0=\{1,2,\ldots,\max F_0\}\subset\N$ such that
  \begin{align*}
    A_{F_0}:=\sup_{j\geq 0,\ x\in\N}\E_x\left(e^{2\lambda T^j_{F_0}}\right)<\infty.
  \end{align*}
 Since $\E_x\left(e^{2\lambda T^j_F}\right)$ is non-increasing in $F$, $A_F\leq A_{F_0}$ for
  all $F\supset F_0$. Given $F\supset F_0$ such that $F=\{1,2,\ldots,\max F\}$, we deduce from~\eqref{eq:S=expect} that
  \begin{align*}
    \alpha(\sigma_j,t) & \leq A_{F_0}+\sqrt{A_{F_0}}\,\sup_{k\geq 0}(S^k_{\max F})^{1/2}\frac{\alpha(s_{j+1},t)}{\sqrt{v_j}}.
  \end{align*}
  We set $\varepsilon=\exp(-C^*-1)$ with $C^*=\sup_{j\in J_\lambda}C_{\lambda,j}<\infty$. We then deduce from~\eqref{eq:assumption1}
  that there exists a finite $F\subset\N$ such that
  \begin{align}
    \label{eq:youpi2}
    \alpha(\sigma_j,t)\leq A_{F_0}+\varepsilon\,\frac{\alpha(s_{j+1},t)}{\sqrt{v_j}}.
  \end{align}
  Combining~\eqref{eq:youpi1} and~\eqref{eq:youpi2}, for all $j\geq 0$,
  \begin{align*}
    \alpha(\sigma_j,t)\leq A_{F_0}+\frac{\varepsilon}{\sqrt{v_j}}e^{\lambda u_{j+1}}\alpha(\sigma_{j+1},t).
  \end{align*}
  A straightforward induction then implies that, for all $n\geq 0$,
  \begin{align*}
    \alpha(\sigma_j,t)\leq A_{F_0}\left[1+\sum_{k=1}^{n} e^{\lambda (u_{j+1}+\ldots+u_{j+k})}\frac{\varepsilon^k}{\sqrt{v_{j}\ldots
        v_{j+k-1}}}\right]+e^{\lambda (u_{j+1}+\ldots+u_{j+n+1})}\frac{\varepsilon^{n+1}\ \alpha(\sigma_{j+n+1},t)}{\sqrt{v_{j}\ldots v_{j+n}}},
  \end{align*}
  and hence, since $\alpha(s,t)\leq e^{\lambda t}$ for all $s\geq 0$,
  \begin{align*}
    \alpha(\sigma_j,t)\leq & A_{F_0}\left[1+\sum_{k=1}^{+\infty} \varepsilon^k\ \exp\left(\sum_{\ell=1}^k\lambda u_{j+\ell}-\frac{\log v_{j+\ell-1}}{2}\right)\right]
    \\ & +\liminf_{n\rightarrow+\infty}\ \varepsilon^{n+1}\ \exp\left(\lambda t+\sum_{\ell=1}^n\lambda u_{j+\ell}-\frac{\log v_{j+\ell-1}}{2}\right).
  \end{align*}
  Assuming that $j$ belongs to the set $J_\lambda$ of Assumption~\eqref{eq:assumtion5}, by definition of $\varepsilon$ and $C^*$, we
  deduce that
  \begin{align*}
    \alpha(\sigma_j,t)\leq A_{F_0}\sum_{k=0}^{+\infty} e^{-k}+\liminf_{n\rightarrow+\infty}\ e^{\lambda
      t-n-1}=\frac{A_{F_0}}{1-1/e}.
  \end{align*}
  Letting $t\rightarrow+\infty$, we finally obtain
  \begin{equation*}
    \sup_{j\in J_\lambda}\ \sup_{x\in\N}\,\E_{\sigma_j,x}\left[e^{\lambda(T_F-\sigma_j)}\right]\leq\frac{A_{F_0}}{1-1/e}. \qedhere
  \end{equation*}
\end{proof}

\begin{proof}[Proof of Lemma~\ref{lem:2}]
  Given $\varepsilon>0$ such that $\P(V\geq\varepsilon)>0$, we can assume without loss of generality that $V\geq\varepsilon>0$ almost
  surely since, otherwise, we may modify the sequences $(u_j,j\geq 0)$ and $(v_j,j\geq 0)$ by removing all the favorable time
  intervals such that $v_j<\varepsilon$ and concatenating them with the surrounding unfavorable intervals. It is easy to check that
  this modifies the sequence $(u_j,v_j)_{j\geq 0}$ as an i.i.d.\ sample of a new random couple $(U',V')$ such that
  $V'\geq\varepsilon$ almost surely, and $\E U'=\E(U\mid V\geq\varepsilon)+\frac{\E(U+V\mid
    V<\varepsilon)}{\P(V\geq\varepsilon)}<\infty$.

  For all $i< j$, we introduce
  \begin{align*}
    S_{i,j}=\frac{1}{j-i}(u_{i+1}+\ldots+u_j).
  \end{align*}
  Since $\E U<\infty$, the strong law of large numbers implies that $S_{i,j}$ converges to $\E U$ when $j\rightarrow+\infty$ for all
  $i\geq 0$ and hence $\sup_{j> i} S_{i,j}<\infty$ almost surely. Therefore, there exists $A>0$ such that
  \begin{align*}
    \P\left(\sup_{j> i} S_{i,j}\leq A\right)>\frac{1}{2},\quad\forall i\geq 0.
  \end{align*}
  Then, for all $k_0\geq 1$,
  \begin{align*}
    \P\left(\sup_{j>i} S_{i,j}\leq A\ \text{and}\ \sup_{j> i+k_0} S_{i+k_0,j}\leq A\right)>0,\quad\forall i\geq 0.
  \end{align*}
  For any given $t_0>0$, we choose $k_0\in\N$ such that $k_0\varepsilon\geq t_0$. There exists a finite constant $C$ such that
  \begin{align*}
    p:=\P\left(\sup_{j> i} S_{i,j}\leq A,\ \sup_{j>i+k_0} S_{i+k_0,j}\leq A\ \text{and}\ v_i+\ldots+v_{i+k_0-1}\leq C\right)>0,\quad\forall i\geq 0.
  \end{align*}

  Now, for all $i\geq 0$ and $n\geq 1$, define
  \begin{align*}
    \Gamma_{i,n}:=
    \begin{cases}
      \left\{S_{i,i+n}> A\text{ or }v_i+\ldots+v_{i+n-1}>C\right\}  & \text{if }n\leq k_0, \\
      \left\{S_{i,i+n}> A\text{ or }S_{i+k_0,i+n}>A\text{ or }v_i+\ldots+v_{i+k_0-1}>C\right\}  & \text{if }n\geq k_0+1
    \end{cases}
  \end{align*}
  and consider the following random sequence
  \begin{align*}
    I_0=0\quad\text{and}\quad I_{k+1}=
    \begin{cases}
      I_k+\inf\{n\geq 1\text{ s.t.\ }\Gamma_{I_k,n}\text{ is satisfied}\} & \text{if }I_k<\infty, \\
      +\infty & \text{otherwise.}
    \end{cases}
  \end{align*}
  Since $\Gamma_{i,n}$ is measurable with respect to $\sigma(u_{i+1},\ldots,u_{i+n},v_i,\ldots,v_{i+n-1})$, the sequence $(I_k,k\geq
  0)$ is a Markov chain in $\N\cup\{+\infty\}$ absorbed at $+\infty$, with independent increments up to absorption. Moreover, at each
  step, the probability of absorption is equal to $p>0$. We deduce that
  \begin{multline*}
    \P\left(\forall i\geq 0,\ \sup_{j>i}S_{i,j}>A\text{ or }\sup_{j>i+k_0}S_{i+k_0,j}>A\text{ or }v_i+\ldots+v_{i+k_0-1}>C\right) \\
    \leq\P((I_k,k\geq 0)\text{ is never absorbed at }+\infty)=0.
  \end{multline*}
  As a consequence, for any fixed $i_0\geq 0$,
  \begin{align*}
    \P\left(\forall i\geq i_0,\ \sup_{j>i}S_{i,j}>A\text{ or }\sup_{j>i+k_0}S_{i+k_0,j}>A\text{ or }v_i+\ldots+v_{i+k_0-1}>C\right)=0,
  \end{align*}
  from which we deduce that
  \begin{align*}
    \P\left(\sup_{j>i}S_{i,j}\leq A,\ \sup_{j>i+k_0}S_{i+k_0,j}\leq A,\ v_i+\ldots+v_{i+k_0-1}\leq C\text{ for infinitely many }i\geq
      0\right)=1.
  \end{align*}
  Since,
  \begin{align*}
    \sigma_{i+k_0}-\sigma_i=v_i+u_{i+1}+v_{i+1}+\ldots+v_{i+k_0-1}+u_{i+k_0},
  \end{align*}
  since $\sup_{j>i}S_{i,j}\leq A$ implies that $u_{i+1}+\ldots+u_{i+k_0}\leq k_0 A$ and since $V\geq\varepsilon$ almost surely, we
  deduce that
  \begin{multline*}
    \P\left(C_{\lambda,i}\leq\lambda A-\frac{\log\varepsilon}{2},\ C_{\lambda,i+k_0}\leq\lambda A-\frac{\log\varepsilon}{2},\ \right. \\
      \text{and }k_0\varepsilon\leq\sigma_{i+k_0}-\sigma_i\leq k_0A+C\text{ for infinitely many }i\geq 0\Bigg)=1.
  \end{multline*}
  In other words, we proved that there exists $A>0$ such that
  \begin{align*}
    J_\lambda:=\{j\geq 0:C_{\lambda,j}\leq A\}
  \end{align*}
  is infinite, and that, for all $t_0>0$, setting $k_0=\lceil t_0/\varepsilon\rceil$,
  \begin{align*}
      \liminf_{j\in J_\lambda,\ j\rightarrow+\infty}\ \inf\left\{\sigma_{k}-\sigma_{j}:\ k\in J_\lambda,\ k>j,\
    \sigma_{k}-\sigma_{j}>t_0\right\}\leq k_0A+C.
  \end{align*}
  This concludes the proof of~\eqref{eq:assumtion5} and~\eqref{eq:assumtion6} and hence of Lemma~\ref{lem:2}.
\end{proof}

\section{Proof of Theorem~\ref{thm:main-result}}
\label{sec:pf-main-thm}

 Some parts of the proof are translations of the ideas of~\cite{champagnat-villemonais-15} in terms of penalized processes. 

\medskip

\noindent\textit{Step 1: control of the normalized distribution after a time 1}\\
Let us show that, for all $s\geq 0$, $T\geq s+1$ and $x_1,x_2\in E$, there exists a measure $\nu^{s,T}_{x_1,x_2}$ with mass greater than $d_{s+1}$ such that, for all non-negative measurable function $f:E\rightarrow\R_+$,
\begin{align}
\label{eq:control-of-distribution}
\delta_{x_i}K_{s,s+1}^T f\geq \nu^{s,T}_{x_1,x_2}(f),\text{ for }i=1,2.
\end{align}
Fix $x_1,x_2\in E$, $i\in\{1,2\}$, $t\geq 1$ and a measurable non-negative function $f: E\rightarrow
\R_+$. Using the Markov property, we have
\begin{align*}
\E_{s,x_i}(f(X_{s+1})Z_{s,T})&=\E_{s,x_i}(f(X_{s+1})Z_{s,s+1}\E_{s+1,X_{s+1}}(Z_{s+1,T}))\\
&\geq \nu_{s+1,x_1,x_2}\left(f(\cdot)\E_{s+1,\cdot}(Z_{s+1,T})\right)\E_{s,x_i}(Z_{s,s+1}),
\end{align*}
by definition of $\nu_{s+1,x_1,x_2}$.
Dividing both sides by $\E_{s,x_i}(Z_{s,T})$, we deduce that
\begin{align*}
\delta_{x_i}K_{s,s+1}^T\left(f\right)
&\geq \nu_{s+1,x_1,x_2}\left(f(\cdot)\E_{s+1,\cdot}(Z_{s+1,T})\right)
\frac{\E_{s,x_i}(Z_{s,s+1})}
{\E_{s,x_i}(Z_{s,T})}.
\end{align*}
But we have
\begin{align*}
\E_{s,x_i}(Z_{s,T})\leq \E_{s,x_i}(Z_{s,s+1})\sup_{y\in E} \E_{s+1,y}(Z_{s+1,T}),
\end{align*}
so that
\begin{align*}
\delta_{x_i}K_{s,s+1}^T f\geq\frac{\nu_{s+1,x_1,x_2}\left(f(\cdot)\E_{s+1,\cdot}(Z_{s+1,T})\right)}{\sup_{y\in E} \E_{s+1,y}(Z_{s+1,T})}.
\end{align*}
Now, by definition of $d_{s+1}$, the non-negative measure
\begin{align*}
\nu^{s,T}_{x_1,x_2}:f\mapsto \frac{\nu_{s,x_1,x_2}\left(f(\cdot)\E_{s+1,\cdot}(Z_{s+1,T})\right)}{\sup_{y\in E} \E_{s+1,y}(Z_{s+1,T})}
\end{align*}
has a total mass greater than $d_{s+1}$. Therefore~\eqref{eq:control-of-distribution} holds.

\bigskip
\noindent
\textit{Step 2: exponential contraction for Dirac initial distributions and proof of~\eqref{eq:main-1}}\\
We now prove that, for all $x,y\in E$ and $0\leq s\leq s+1\leq t \leq T$
\begin{align}
\label{eq:thm1-step2}
\left\|\delta_{x} K_{s,t}^T-\delta_{y} K_{s,t}^T\right\|_{TV}\leq 2 \prod_{k=0}^{\lfloor t-s \rfloor -1}\left(1-d_{t-k}\right).
\end{align}

We deduce from~\eqref{eq:control-of-distribution} that, for all $x_1,x_2\in E$,
 \begin{align*}
   \left\|\delta_{x_1} K_{s,s+1}^T-\delta_{x_2} K_{s,s+1}^T \right\|_{TV}
   &\leq
   \left\|\delta_{x_1} K_{s,s+1}^T -  \nu^{s,T}_{x_1,x_2}\right\|_{TV}
   +\left\|\delta_{x_2} K_{s,s+1}^T - \nu^{s,T}_{x_1,x_2}\right\|_{TV}\\
   &\leq
   2(1-d_s).
 \end{align*}
It is then standard (see e.g.~\cite{champagnat-villemonais-15}) to deduce that, for any probability measures $\mu_1$ and $\mu_2$ on $E$,

 \begin{equation*}
   \left\|\mu_1 K_{s,s+1}^T-\mu_2 K_{s,s+1}^T \right\|_{TV}\leq (1-d_s) \|\mu_1-\mu_2\|_{TV}.
 \end{equation*}
 Using the semi-group property of $(K_{s,t}^T)_{s,t}$,
 we deduce that, for any $x,y\in E$,
 \begin{align*}
   \left\|\delta_x K_{s,t}^T - \delta_y K_{s,t}^T\right\|_{TV}
   &=\left\|\delta_x K^T_{s,t-1} K_{t-1,t}^T - \delta_y K_{0,t-1}^T K_{t-1,t}^T\right\|_{TV}\\
           &\leq \left(1-d_{t}\right)\left\|\delta_x K_{s,t-1}^T - \delta_y K_{s,t-1}^T\right\|_{TV}\\
           &\leq\ \ldots\ \leq \prod_{k=0}^{\lfloor t-s\rfloor-1}\left(1-d_{t-k}\right)\, \left\|\delta_x K_{s,t-\lfloor t-s\rfloor}^T - \delta_y
             K_{s,t-\lfloor t-s\rfloor}^T\right\|_{TV}\\
           & \leq  2 \prod_{k=0}^{\lfloor t-s\rfloor-1}\left(1-d_{t-k}\right).
 \end{align*}
One deduces~\eqref{eq:main-1} with standard arguments as above.

\bigskip
\noindent
\textit{Step 3: exponential contraction for general initial distributions}\\
We prove now that 
for any pair of initial probability measures $\mu_1,\mu_2$ on $E$, for all $0\leq s\leq s+1\leq t \leq T\geq 0$,
\begin{align}
\label{eq:thm1-step3}
\left\|\frac{\E_{s,\mu_1}(\11_{X_t\in \cdot}Z_{s,T})}{\E_{s,\mu_1}(Z_{s,T})}-\frac{\E_{s,\mu_2}(\11_{X_t\in \cdot}Z_{s,T})}{\E_{s,\mu_2}(Z_{s,T})}\right\|_{TV}\leq 2 \prod_{k=0}^{\lfloor t-s \rfloor -1}\left(1-d_{t-k}\right).
\end{align}
Taking $t=T$ then entails~\eqref{eq:main-2} and ends the proof of Theorem~\ref{thm:main-result}.

Let $\mu_1$ be a probability measure on $E$ and $x\in E$. We have
\begin{align*}
  &\left\|\frac{\E_{s,\mu_1}(\11_{X_t\in \cdot}Z_{s,T})}{\E_{s,\mu_1}(Z_{s,T})}-\frac{\E_{s,x}(\11_{X_t\in \cdot}Z_{s,T})}{\E_{s,x}(Z_{s,T})}\right\|_{TV}\\
  &=\frac{1}{\E_{s,\mu_1}(Z_{s,T})}\left\|\E_{s,\mu_1}(\11_{X_t\in \cdot}Z_{s,T})-\E_{s,\mu_1}(Z_{s,T})\delta_x K_{s,t}^T \right\|_{TV}\\
  &\leq
\frac{1}{\E_{s,\mu_1}(Z_{s,T})} \int_{y\in E}
  \left\|\E_{s,y}(\11_{X_t\in \cdot}Z_{s,T})-\E_{s,y}(Z_{s,T})\delta_x K_{s,t}^T \right\|_{TV}d\mu_1(y)\\
  &\leq   \frac{1}{\E_{s,\mu_1}(Z_{s,T})} \int_{y\in E}
  \E_{s,y}(Z_{s,T})\left\|\delta_y K_{s,t}^T-\delta_x K_{s,t}^T\right\|_{TV}d\mu_1(y)\\
  &\leq  \frac{1}{\E_{s,\mu_1}(Z_{s,T})} \int_{y\in E}
  \E_{s,y}(Z_{s,T})\ 2 \prod_{k=0}^{\lfloor t-s \rfloor -1}\left(1-d_{t-k}\right) d\mu_1(y)\\
  &\leq 2 \prod_{k=0}^{\lfloor t-s \rfloor -1}\left(1-d_{t-k}\right).
\end{align*}
The same computation, replacing $\delta_x$ by any probability measure,
leads to~\eqref{eq:thm1-step3}.

\section{Proof of Proposition~\ref{prop:eta} and Corollary~\ref{cor:unicite-EDP}}
\label{sec:pf-prop}

\subsection{Proof of Proposition~\ref{prop:eta}}

Fix $s\geq 0$. Let us first prove~\eqref{eq:prop-borne-eta}. Note that, if $d'_v=0$ for all $v\geq s+1$, there is nothing to prove,
so let us assume the converse. Fix $t\geq s+1$ such that $d'_t>0$. Then the measure $\nu_t$ is positive and we define for all $x\in
E$ and $u\geq t$
\begin{align*}
\eta_{t,u}(x)=\frac{\E_{t,x}\left(Z_{t,u}\right)}{\E_{t,\nu_{t}}(Z_{t,u})}.
\end{align*}
For all $u\geq t$ and $x,y\in E$, we have
\begin{align*}
\frac{\E_{s,x}(Z_{s,u})}{\E_{s,y}(Z_{s,u})} 
&=\frac{\E_{s,x}\left(Z_{s,t}\E_{t,X_t}(Z_{t,u})\right)}{\E_{s,y}(Z_{s,t}\E_{t,X_t}(Z_{t,u}))} 
\\
&=\frac{\Phi_{s,t}(\delta_x)(\EE_{t,\cdot}(Z_{t,u}))}{\Phi_{s,t}(\delta_y)(\EE_{t,\cdot}(Z_{t,u}))}\,\frac{\E_{s,x}(Z_{s,t})}{\E_{s,y}(Z_{s,t})} \\
&=\frac{\Phi_{s,t}(\delta_x)(\eta_{t,u})}{\Phi_{s,t}(\delta_y)(\eta_{t,u})}\,\frac{\E_{s,x}(Z_{s,t})}{\E_{s,y}(Z_{s,t})}.
\end{align*}
Therefore
\begin{align}
  \left|\frac{\E_{s,x}(Z_{s,u})}{\E_{s,y}(Z_{s,u})}-\frac{\E_{s,x}(Z_{s,t})}{\E_{s,y}(Z_{s,t})}\right| & =\frac{\E_{s,x}(Z_{s,t})}{\E_{s,y}(Z_{s,t})}\ 
\frac{|\Phi_{s,t}(\delta_x)(\eta_{t,u})-\Phi_{s,t}(\delta_y)(\eta_{t,u})|}{\Phi_{s,t}(\delta_y)(\eta_{t,u})} \notag \\
&\leq \frac{\E_{s,x}(Z_{s,t})}{\E_{s,y}(Z_{s,t})}\ \frac{\|\eta_{t,u}\|_\infty}{\Phi_{s,t}(\delta_y)(\eta_{t,u})}\prod_{k=0}^{\lfloor
  t-s \rfloor-1}\left(1-d_{t-k}\right),
\label{eq:pf-prop-eta-1.25}
\end{align}
where we used the bound~\eqref{eq:main-2} of Theorem~\ref{thm:main-result} in the last inequality.

Let us first prove that $\eta_{t,u}$ is uniformly bounded and that we have $\Phi_{s,t}(\mu)(\eta_{t,u})\geq 1$ for all positive
measure $\mu$ on $E$. First, by definition of $d'_t$, we have
\begin{align}
 \label{eq:pf-prop-eta-2}
\eta_{t,u}(x)&=\frac{\E_{t,x}\left(Z_{t,u}\right)}{\E_{t,\nu_{t}}(Z_{t,u})}\leq 1/d'_t.
\end{align}
Second, by Markov's property,
\begin{align*}
\Phi_{s,t}(\mu)(\eta_{t,u})&=\frac{\E_{s,\mu}(Z_{s,t}\eta_{t,u}(X_t))}{\E_{s,\mu}(Z_{s,t})}\\
&= \frac{\E_{s,\mu}(Z_{s,u})}{\E_{s,\mu}(Z_{s,t})\E_{t,\nu_t}(Z_{t,u})},
\end{align*}
where, using the definition of $\nu_t$,
\begin{align*}
\E_{s,\mu}(Z_{s,u})&=\E_{s,\mu}\left[Z_{s,t-1}\E_{t-1,X_{t-1}}(Z_{t-1,t}\E_{t,X_t}(Z_{t,u}))\right]\\
&= \E_{s,\mu}\Big\{Z_{s,t-1}\Phi_{t-1,t}(\delta_{X_{t-1}})\left[\E_{t,\cdot}(Z_{t,u})\right]\E_{t-1,X_{t-1}}(Z_{t-1,t})\Big\}\\
&\geq \E_{s,\mu}\left[Z_{s,t-1}\E_{t,\nu_t}(Z_{t,u})\E_{t-1,X_{t-1}}(Z_{t-1,t})\right]\\
&=  \E_{s,\mu}(Z_{s,t})\E_{t,\nu_t}(Z_{t,u}).
\end{align*}
Hence,
\begin{align}
 \label{eq:pf-prop-eta-3}
\Phi_{s,t}(\mu)(\eta_{t,u})\geq 1.
\end{align}

Now, let $t_1$ be the smallest $v\geq s+1$ such that $d'_{v_1}>0$. Using a similar computation as in the proof
of~\eqref{eq:pf-prop-eta-3} above, we have
\begin{align}
  \frac{\E_{s,x}(Z_{s,u})}{\E_{s,y}(Z_{s,u})}
  & =\frac{\E_{s,x}[Z_{s,t_1-1}\Phi_{t_1-1,t_1}(\delta_{X_{t_1-1}})(\E_{t_1,\cdot}(Z_{t_1,u})) \E_{t_1-1,X_{t_1-1}}(Z_{t_1-1,t_1})]}{\E_{s,y}[Z_{s,t_1}\E_{t_1,X_{t_1}}(Z_{t_1,u})]}
  \notag \\
  & \geq \frac{\E_{t_1,\nu_{t_1}}(Z_{t_1,u}))}{\sup_{z\in E}\E_{t_1,z}(Z_{t_1,u})}\ \frac{\E_{s,x}(Z_{s,t_1})}{\E_{s,y}(Z_{s,t_1})}
  \notag \\
  & \geq d'_{t_1}\frac{\E_{s,x}(Z_{s,t_1})}{\sup_{z\in E}\E_{s,z}(Z_{s,t_1})}, \label{eq:preuve-prop}
\end{align}
where we used the definition of $d'_{t_1}$ in the last inequality. Note that the right-hand side of~\eqref{eq:preuve-prop} does not
depend on $u$ and $y$ and is positive by~\eqref{eq:hyp-first}.

Inserting the inequalities~\eqref{eq:pf-prop-eta-2},~\eqref{eq:pf-prop-eta-3} and~\eqref{eq:preuve-prop}
in~\eqref{eq:pf-prop-eta-1.25}, we obtain
\begin{align*}
  \left|\frac{\E_{s,x}(Z_{s,u})}{\E_{s,y}(Z_{s,u})}-\frac{\E_{s,x}(Z_{s,t})}{\E_{s,y}(Z_{s,t})}\right| & \leq \frac{\sup_{z\in
      E}\E_{s,z}(Z_{s,t_1})}{d'_{t_1}\E_{s,y}(Z_{s,t_1})}\ \frac{1}{d'_t}\prod_{k=0}^{\lfloor
    t-s \rfloor-1}\left(1-d_{t-k}\right) \\
  & =C_{s,y}\frac{1}{d'_t}\prod_{k=0}^{\lfloor t-s \rfloor-1}\left(1-d_{t-k}\right)
\end{align*}
where $C_{s,y}$ only depends on $s$ and $y$.

To complete the proof of~\eqref{eq:prop-borne-eta}, it remains to observe that, for any $u\geq t\geq s+1$ (not necessarily such that
$d'_t>0$) and for all $v\in[s+1,t]$ such that $d'_{v}>0$, we have $t_1\leq v$ and hence we can apply the last inequality to obtain
\begin{align*}
  \left|\frac{\E_{s,x}(Z_{s,u})}{\E_{s,y}(Z_{s,u})}-\frac{\E_{s,x}(Z_{s,t})}{\E_{s,y}(Z_{s,t})}\right| & \leq
  \left|\frac{\E_{s,x}(Z_{s,v})}{\E_{s,y}(Z_{s,v})}-\frac{\E_{s,x}(Z_{s,t})}{\E_{s,y}(Z_{s,t})}\right|+\left|\frac{\E_{s,x}(Z_{s,v})}{\E_{s,y}(Z_{s,v})}-\frac{\E_{s,x}(Z_{s,u})}{\E_{s,y}(Z_{s,u})}\right| \\
  & \leq 2C_{s,y}\frac{1}{d'_v}\prod_{k=0}^{\lfloor v-s \rfloor-1}\left(1-d_{v-k}\right).
\end{align*}

Now, we assume that~\eqref{eq:hyp-prop-eta} holds true. We fix $x_0\in E$. It follows from~\eqref{eq:prop-borne-eta} that
$x\mapsto\frac{\E_{s,x}(Z_{s,t})}{\E_{s,x_0}(Z_{s,t})}$ converges uniformly when $t\rightarrow+\infty$ to some function $\eta_s$,
which is positive because of~\eqref{eq:preuve-prop}.

Moreover, for all $s\leq t\leq u$, 
\begin{align}
\E_{s,x}\left(Z_{s,t}\frac{\E_{t,X_t}(Z_{t,u})}{\E_{t,x_0}(Z_{t,u})}\right) & =\frac{\E_{s,x}(Z_{s,u})}{\E_{t,x_0}(Z_{t,u})} \notag \\
& =\frac{\E_{s,x}(Z_{s,u})}{\E_{s,x_0}(Z_{s,u})}\ \E_{s,x_0}\left(Z_{s,t}\frac{\E_{t,X_t}(Z_{t,u})}{\E_{t,x_0}(Z_{t,u})}\right).
\label{eq:pf-prop-eta-4}
\end{align}
For all probability measure $\mu$ on $E$, integrating both sides of the equation with respect to $\mu$, letting $u\rightarrow\infty$
and using Lebesgue's theorem, we deduce that, for all $s\leq t\in I$, there exists a positive constant $c_{s,t}$ which does not
depend on $\mu$ such that
\begin{align*}
c_{s,t}=\frac{\E_{s,\mu}(Z_{s,t}\eta_t(X_t))}{\mu(\eta_s)}.
\end{align*}
In addition, for all $s\leq t\leq u\in I$,
\begin{align*}
  c_{s,t}c_{t,u} & =\frac{\E_{s,x}(Z_{s,t}\eta_t(X_t))}{\eta_s(x)}\,\frac{\E_{t,\mu}(Z_{t,u}\eta_u(X_u))}{\mu(\eta_t)}.
\end{align*}
Choosing the probability measure $\mu$ defined by $\mu(f)=\frac{\E_{s,x}(Z_{s,t}f(X_t))}{\E_{s,x}(Z_{s,t})}$ for all bounded
measurable $f$ and using Markov's property, we obtain
\begin{align*}
  c_{s,t}c_{t,u} & =\frac{\E_{s,x}(Z_{s,t})\,\E_{t,\mu}(Z_{t,u}\eta_u(X_u))}{\eta_s(x)}=\frac{\E_{s,x}(Z_{s,u}\eta_u(X_u))}{\eta_s(x)}=c_{s,u}.
\end{align*}
Because of the last equality, replacing for all $s\geq 0$ the function $\eta_s(x)$ by $\eta_s(x)/c_{0,s}$
entails~\eqref{eq:fonction-propre}.

\subsection{Proof of Corollary~\ref{cor:unicite-EDP}}

Let $(f_s)_{s\geq 0}$ be a solution of~\eqref{eq:valeur-propre-sg} satisfying~\eqref{eq:condition}. Fix $x_0\in E$ and for all $s\geq
0$, let $\nu_s=\E_{s,x_0}(Z_{0,s})$. Using~\eqref{eq:condition} and applying~\eqref{eq:main-2} with $\mu_1=\delta_x$ and
$\mu_2=\nu_s$, we have for all $s\geq 0$, $x\in E$ and for $t\rightarrow+\infty$,
\begin{align*}
  f_s(x) & =\E_{s,x}(Z_{s,t}f_t(X_t))\sim \E_{s,x}(Z_{s,t})\frac{\E_{0,x_0}(Z_{0,t}f_t(X_t))}{\E_{s,\nu_s}(Z_{s,t})}.
\end{align*}
Using~\eqref{eq:limite-eta} (integrated with respect to $\nu_s(dx)$), we deduce
\begin{align*}
  f_s(x) & \sim \frac{\eta_s(x)}{\nu_s(\eta_s)}\,\E_{s,\nu_s}(Z_{s,t})\,\frac{f_0(x_0)}{\E_{s,\nu_s}(Z_{s,t})} \\
  & \sim\frac{\eta_s(x)}{\eta_0(x_0)}\, f_0(x_0).
\end{align*}
Since both sides are independent of $t$, we obtain
\begin{align*}
  f_s=\eta_s\,\frac{f_0(x_0)}{\eta_0(x_0)}.
\end{align*}

\section{Proof of Theorem~\ref{thm:Q-proc}}
\label{sec:pf-Q-proc}

Let us define the probability measure $Q_{s,x}^t$ by
\begin{align*}
dQ_{s,x}^t&=\frac{Z_{s,t}}{\E_{s,x}(Z_{s,t})}d\P_{s,x},\quad\text{on }\mathcal{F}_{s,t}
\end{align*}
We have, for all $0\leq s\leq u\leq t$,
\begin{align*}
\frac{\E_{s,x}(Z_{s,t}\mid \cF_{s,u})}{\E_{s,x}(Z_{s,t})}=\frac{Z_{s,u}\E_{u,X_u}(Z_{u,t})}{\E_{s,x}[Z_{s,u}\E_{u,X_u}(Z_{u,t})]}
=\frac{Z_{s,u}\eta_{u,t}(X_u)}{\E_{s,x}[Z_{s,u}\eta_{u,t}(X_u)]}.
\end{align*}
By Proposition~\ref{prop:eta}, this converges almost surely when $t\rightarrow\infty$ to
\begin{align*}
M_{s,u}:=\frac{Z_{s,u}\eta_u(X_u)}{\E_{s,x}[Z_{s,u}\eta_u(X_u)]},
\end{align*}
where $\E_{s,x}(M_{s,u})=1$.

By the penalisation's theorem of Roynette, Vallois and Yor~\cite[Theorem~2.1]{Roynette_Vallois_Yor_2006}, these two conditions
(almost sure convergence and $\E_{s,x}(M_{s,u})=1$) imply that $(M_{s,t},t\geq s)$ is a martingale under $\P_{s,x}$ and that
$Q_{s,x}^{t}(\Lambda_{s,u})$ converges to $\E_{s,x}\left(M_{s,u}\11_{\Lambda_{s,u}}\right)$ for all $\Lambda_{s,u}\in{\cal F}_{s,u}$ when
$t\rightarrow\infty$. This means that $\Q_{s,x}$ is well defined and
\begin{align*}
\restriction{\frac{d\Q_{s,x}}{d\P_{s,x}}}{{\cal F}_{s,u}}=M_{s,u}.
\end{align*}

Let us now prove that the family $(\Q_{s,x})_{s\in I,x\in E}$ defines a time inhomogeneous Markov process, that is
for all $s\leq u\leq t$, all $x\in E$ and all positive measurable function $f$,
\begin{align*}
\E_{\Q_{s,x}}(f(X_t)\mid \cF_{s,u})=\E_{\Q_{u,X_u}}(f(X_t)).
\end{align*}
We easily check from the definition of the conditional expectation that
\begin{align*}
M_{s,u}\E_{\Q_{s,x}}(f(X_t)\mid \cF_{s,u})&=\E_{s,x}\left(M_{s,t}f(X_t)\mid \cF_{s,u}\right)\\
      &=\frac{\E_{s,x}[Z_{s,t}\eta_t(X_t)f(X_t)\mid\mathcal{F}_{s,u}]}{\E_{s,x}(Z_{s,t}\eta_t(X_t))}\\
      &=\frac{Z_{s,u}\E_{u,X_u}(Z_{u,t}\eta_t(X_t))}{\E_{s,x}(Z_{s,t}\eta_t(X_t))}\E_{u,X_u}\left(\frac{Z_{u,t}\eta_t(X_t)}{\E_{u,X_u}(Z_{u,t}\eta_t(X_t))}f(X_t)\right)\\
      &=\frac{Z_{s,u}\E_{u,X_u}(Z_{u,t}\eta_t(X_t))}{\E_{s,x}(Z_{s,t}\eta_t(X_t))}\E_{\Q_{u,X_u}}\left(f(X_t)\right),
\end{align*}
where we used the Markov property of $X$ under $\P_{s,x}$, the fact that $Z_{s,t}=Z_{s,u}Z_{u,t}$ and the definition of $\Q_{u,X_u}$.
Using the above equality with $f=1$, we conclude that
\begin{align*}
  \frac{Z_{s,u}\E_{u,X_u}(Z_{u,t}\eta_t(X_t))}{\E_{s,x}(Z_{s,t}\eta_t(X_t))}=M_{s,u}
\end{align*}
(we could also use~\eqref{eq:pf-prop-eta-4}). Hence, the Markov property holds for $(\Q_{s,x})_{s\in I,x\in E}$.

The inequality~\eqref{eq:thm-Q-proc-3} is a direct consequence of~\eqref{eq:main-1} in Theorem~\ref{thm:main-result}.

\bibliographystyle{abbrv}
\bibliography{biblio-bio,biblio-denis,biblio-math,biblio-math-nicolas}

\end{document}